\documentclass[12pt]{amsart}

\usepackage[left=3cm,top=2.5cm,bottom=2.5cm,right=3cm]{geometry}

\usepackage{times}
\usepackage{amssymb}
\usepackage{graphicx,xspace}
\usepackage{epsfig}
\usepackage{enumitem}
\usepackage[usenames,dvipsnames]{xcolor}
\usepackage{tikz}
\usepackage{gensymb}
\usepackage[T1]{fontenc}
\usepackage[utf8]{inputenc}
\usepackage{bm}
\usepackage{wrapfig}
\usepackage[config, labelfont={normalsize}]{caption,subfig}
\usepackage{mathrsfs}
\definecolor{green}{RGB}{0,127,0}
\definecolor{red}{RGB}{191,0,0}
\usepackage[colorlinks,cite color=red,link color=green,pagebackref=true]{hyperref}

\usepackage{todonotes}

\usetikzlibrary{matrix,arrows,calc}
\usetikzlibrary{decorations.markings}
\usetikzlibrary{patterns}
\usetikzlibrary{snakes}
\usetikzlibrary{shapes}

\usepackage[capitalize]{cleveref}

\theoremstyle{plain}

\newtheorem{lemma}{Lemma}[section]
\newtheorem{theorem}[lemma]{Theorem}
\newtheorem{corollary}[lemma]{Corollary}
\newtheorem{proposition}[lemma]{Proposition}

\newtheorem{question}[lemma]{Question}
\newtheorem{conjecture}[lemma]{Conjecture}

\theoremstyle{remark}
\newtheorem*{remark}{Remark}

\newtheorem{definition}[lemma]{Definition}

\newcommand{\N}{\mathbb{N}}

\newcommand{\Sym}[1]{\mathfrak{S}_{#1}}

\newcommand{\M}{\mathcal{M}}

\newcommand{\QQ}{\mathbb{Q}}

\newcommand{\OOO}{\mathcal{O}}

\newcommand{\xx}{\bm{x}}

\newcommand{\qq}{\mathfrak{q}}

\newcommand{\Y}{\mathcal{Y}}

\def\la{\lambda}

\def\a{\alpha}
\def\si{\sigma}

\newcommand{\yy}{\bm{y}}
\newcommand{\zz}{\bm{z}}

\DeclareMathOperator{\Sp}{Sp}

\DeclareMathOperator{\Symm}{Sym}

\author[M.~Dołęga]{Maciej Dołęga}
\address{
Wydział Matematyki i Informatyki, 
Uniwersytet im.~Adama Mickiewicza, 
Collegium Mathematicum,
Umultowska 87, 
61-614 Poznań, 
Poland, \newline \indent Instytut Matematyczny,
Uniwersytet Wrocławski,  \mbox{pl.\ Grunwaldzki~2/4,} 50-384
Wrocław, Poland}
\email{maciej.dolega@amu.edu.pl}

 \thanks{
MD is supported from {\it NCN}, grant UMO-2015/16/S/ST1/00420.}



\title[$b$-conjecture for one-face maps]{Top degree part in $b$-conjecture for unicellular bipartite maps}

\begin{document}

\maketitle

\begin{abstract}
Goulden and Jackson (1996) introduced, using Jack symmetric functions, some multivariate generating series $\psi(\bm{x}, \bm{y}, \bm{z}; 1, 1+\beta)$ with an additional parameter $\beta$ that may be interpreted as a continuous deformation of the rooted bipartite maps generating series.
Indeed, it has the property that for $\beta \in \{0,1\}$,
it specializes to the rooted, orientable (general, i.e.~orientable or not, respectively) bipartite maps generating series.
They made the following conjecture:
coefficients of $\psi$ are polynomials in $\beta$ with positive integer coefficients that can be written as a multivariate generating series of rooted, general bipartite maps,
where the exponent of $\beta$ is an integer-valued statistic
that in some sense ``measures the non-orientability'' of
the corresponding bipartite map.

We show that except for two special values of $\beta = 0,1$ for which the
combinatorial interpretation of the coefficients of $\psi$ is known,
there exists a third special value $\beta = -1$ for which the
coefficients of  $\psi$ indexed by two partitions $\mu,\nu$, and one
partition with only one part are given by rooted, orientable bipartite
maps with arbitrary face degrees and black/white vertex degrees given
by $\mu$/$\nu$, respectively. We show that this evaluation
corresponds, up to a sign, to a top-degree part of the coefficients of
$\psi$. As a consequence, we introduce a collection of integer-valued
statistics of maps $(\eta)$ such that the top-degree of the
multivariate generating series of rooted bipartite maps with only one
face (called \emph{unicellular}) with respect to $\eta$ gives the
top-degree of the appropriate coefficients of $\psi$. Finally, we show
that the $b$-conjecture holds true for all rooted, unicellular bipartite maps of genus at most $2$.
\end{abstract}

\section{Introduction}
\subsection{Maps}

Roughly speaking, a \emph{map} is a graph drawn on a certain topological
surface such as the sphere or the Klein bottle (see Section~\ref{subsect:Maps} for precise definitions). This simple object carries both combinatorial and
geometric informations so 
it turned out that maps appear naturally in many
different contexts. In particular, they have deep connections with various branches of discrete mathematics, algebra, or physics (see e.g.~\cite{LandoZvonkin2004, Eynard:book} and references therein). 
One of the major steps in the study of maps is developing methods for
their enumeration (either by generating functions, matrix integral
techniques, algebraic combinatorics, or bijective methods) which is
now a well established domain on its own. Moreover, in many areas in
mathematics and physics map enumeration is crucial. We refer
to \cite{Schaeffer2015} -- a great
introductory text on the enumeration of maps -- which shows how studying enumerative properties of maps is of great
importance in many different contexts. 

In this paper we will focus on an interesting connection,
explored mainly by Goulden and Jackson \cite{GouldenJackson1996},
between map enumeration and symmetric functions theory. This connection led to a twenty years old open problem, called \emph{$b$-conjecture}, that will be the main subject of this paper.

\subsection{Enumeration of bipartite maps in terms of Jack polynomials}

We define a \emph{map} as a connected graph embedded into a \emph{surface} (i.~e.~compact, connected $2$-manifold without boundary) in a way that the \emph{faces} (connected components of the complement of the
graph) are simply connected. A \emph{hypermap} (and, by duality,
\emph{bipartite map}) is a face two-colored map (vertex two-colored
map), so that each edge separates faces (vertices) of different
colors. A map is rooted by distinguishing \emph{the root}, that is the
unique side and the beginning
of the selected edge. A rooted hypermap (bipartite map, by duality)
has the black root face (vertex, respectively), by convention, where root
face (vertex, respectively) is the unique face (vertex, respectively)
incident to the root. An example of a bipartite map is illustrated in~\cref{fig:MapExampleOrient}.

\newcommand{\faceA}{red!50}
\newcommand{\faceB}{blue}
\newcommand{\faceAfill}{red!10}
\newcommand{\faceBfill}{blue!20}

\begin{figure}
\centering
\begin{tikzpicture}[scale=0.6,
black/.style={circle,thick,draw=black,fill=white,inner sep=4pt},
white/.style={circle,draw=black,fill=black,inner sep=4pt},
connection/.style={draw=black,black,auto},
faceAs/.style={\faceA, dashed,  line width=6pt},
]
\scriptsize

\begin{scope}
\clip (0,0) rectangle (10,10);

\fill[pattern color=\faceAfill,pattern=north west lines] (0,0) rectangle (10,10);
\fill[pattern color=\faceBfill, pattern=north east lines] (5,5) rectangle (8,8);

\draw (2,7)  node (b1)     [black] {};
\draw (12,3) node (b1prim) [black] {};
\draw (3,2)  node (w1) [white] {};
\draw (-7,8) node (w1prim) [white] {};

\draw (5,5)  node (AA)     [black] {};
\draw (8,5)  node (BA)     [white] {};
\draw (5,8)  node (AB)     [white] {};
\draw (8,8)  node (BB)     [black] {};
\draw (3,9)  node (C)     [black] {};

\draw[connection]         (w1)      to node {$$} node [swap] {$$} (AA);
\draw[line width=2pt,-left to]         (w1)      to node [pos=0.6] {$$} node [swap,pos=0.6] {$$} (3.5,2.75);
\draw[connection]         (AA)      to node {$$} node [swap] {$$} (AB);
\draw[connection]         (AB)      to node {$$} node [swap] {$$} (BB);
\draw[connection]         (BB)      to node {$$} node [swap] {$$} (BA);
\draw[connection]         (BA)      to node {$$} node [swap] {$$} (AA);
\draw[connection]         (AB)      to node {$$} node [swap] {$$} (C);

\draw[connection]         (w1)      to node [pos=0.5] {$$} node [swap,pos=0.425] {$$} (b1prim);
\draw[connection]         (w1prim)  to node [pos=0.87] {$$} node [swap,pos=0.95] {$$} (b1);
\draw[connection]         (w1)      to node [pos=0.625] {$$}  node [swap,pos=0.5] {$$} (b1);
%\draw[faceAs,-left to]    (w1)      to (AA);

\end{scope}

\draw[very thick,decoration={
    markings,
    mark=at position 0.666  with {\arrow{>}}},
    postaction={decorate}]  
(0,0) -- (10,0);

\draw[very thick,decoration={
    markings,
    mark=at position 0.666  with {\arrow{>}}},
    postaction={decorate}]  
(10,10) -- (0,10);

\draw[very thick,decoration={
    markings,
    mark=at position 0.666  with {\arrow{>>}}},
    postaction={decorate}]  
(10,0) -- (10,10);

\draw[very thick,decoration={
    markings,
    mark=at position 0.666  with {\arrow{>>}}},
    postaction={decorate}]  
(0,10) -- (0,0);

\end{tikzpicture}
\caption{Example of a \emph{rooted, bipartite non-orientable map} with
  the face distribution $\tau = (12,4)$, the black vertex distribution
  $\nu = (3,3,2)$, and the white vertex distribution $\mu =
  (3,2,2,1)$. Faces are indicated by blue and red shaded regions and
  the root is indicated by a thick beginning of the selected edge with a distinguished side. The map is drawn on a projective plane: the left side of the square should be glued to the right side,
as well as bottom to top, as indicated by the arrows. }
\label{fig:MapExampleOrient}
\end{figure}
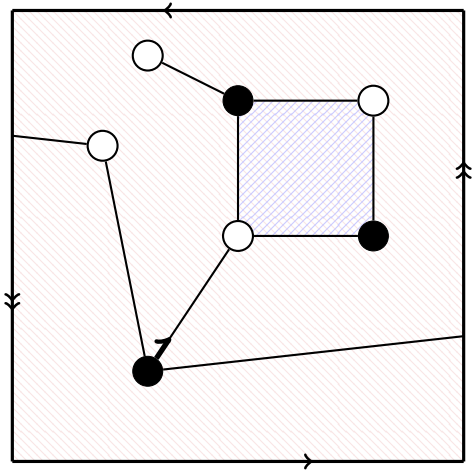

Let $m^\tau_{\mu,\nu}$ ($\widetilde{m}^\tau_{\mu,\nu}$, respectively) 
be the number of rooted hypermaps on orientable (all, respectively)
surfaces, such that $\mu$ lists the degrees of black faces, $\nu$ lists the degrees of white faces and
$\tau$ lists the degrees divided by two of vertices (since a map is face two-colored, all vertex degrees are even numbers). By duality $m^\tau_{\mu,\nu}$ ($\widetilde{m}^\tau_{\mu,\nu}$, respectively) 
is also the number of rooted bipartite maps on orientable (all, respectively)
surfaces, such that $\mu$ lists the degrees of black vertices (we say $\mu$ is a \emph{black vertex distribution}), $\nu$ lists the degrees of white vertices ($\nu$ is a \emph{white vertex distribution}) and
$\tau$ lists the degrees divided by two of faces ($\tau$ is a \emph{face distribution}). 
As standard in enumerative combinatorics,
we will consider the multivariate generating series (m.g.s, for short) for these objects:

\begin{equation}
\label{eq:OrientableGF}
M(\bm{x}, \bm{y}, \bm{z}; t) = \sum_{n \geq 1}t^n\sum_{|\tau| = |\mu|= |\nu| = n} m^\tau_{\mu,\nu} p_\tau(\bm{x}) p_\mu(\bm{y}) p_\nu(\bm{z}),
\end{equation}

\begin{equation}
\label{eq:NonOrientableGF}
\widetilde{M}(\bm{x}, \bm{y}, \bm{z}; t) = \sum_{n \geq 1}t^n\sum_{|\tau| = |\mu|= |\nu| = n} \widetilde{m}^\tau_{\mu,\nu} p_\tau(\bm{x}) p_\mu(\bm{y}) p_\nu(\bm{z}),
\end{equation}

where $p_\tau(\bm{x})$ is a \emph{power-sum symmetric function}, i.e.
\[ p_\tau(\bm{x}) = \prod_i p_{\tau_i}(\bm{x}), \qquad p_k(\bm{x}) =
x_1^k+x_2^k+\cdots \text{ for } k \geq 1,\]
and $|\mu|:=\mu_1+\mu_2+\cdots$ denotes the size of the list $\mu$.
The use of the power-sum symmetric functions as formal variables in the above m.g.s is justified by the remarkable relation, explored by Goulden and Jackson, between bipartite maps enumeration, and symmetric functions theory. Let $J_\la^{(\a)}(\xx)$ be the Jack symmetric function indexed by a partition $\la$
in the infinite alphabet $\xx$ and let $\langle \cdot,
\cdot\rangle_\a$ be the $\a$-deformation of the Hall scalar product on
the space of symmetric functions (see \cref{subsec:Jack} for a precise definition). We also use $\Y$ for the set of all integer partitions. 
% and $|\la|$ for the size of a partition $\la$
Goulden and Jackson defined in their article \cite{GouldenJackson1996} a family of coefficients
$\left(h_{\mu,\nu}^{\tau}(\a-1)\right)_{\mu,\nu,\tau}$ by the following equation:
\begin{multline}
\psi(\bm{x}, \bm{y}, \bm{z}; t, \a) := \a t\frac{\partial}{\partial_t} \log \left( 
\sum_{\la \in \Y} \frac{
J_\la^{(\a)}(\xx) \, J_\la^{(\a)}(\yy) \, J_\la^{(\a)}(\zz)
\, t^{|\la|}}
{\langle J_\la, J_\la \rangle_\a}
\right) \\
 = \sum_{n \ge 1} t^n \left( 
 \sum_{\mu,\nu,\tau \vdash n} h_{\mu,\nu}^{\tau}(\a-1) \,
p_\mu(\xx) \, p_\nu(\yy) \, p_\tau(\zz) \right), 
\label{EqDefH}
\end{multline}
where $\mu,\nu,\tau \vdash n$ means that $\mu$, $\nu$ and $\tau$
are three partitions of size $n$.

This rather involved definition is motivated by the below described combinatorial
interpretations of $\psi$ for two particular values of $\a$.

\begin{theorem}[\cite{JacksonVisentin1990, GouldenJackson1996a}]
\label{theo:MapSeries}
The following equalities of the m.g.s. hold true:
\[ M(\bm{x}, \bm{y}, \bm{z}; t) = \psi(\bm{x}, \bm{y}, \bm{z}; t, 1), \qquad \widetilde{M}(\bm{x}, \bm{y}, \bm{z}; t) = \psi(\bm{x}, \bm{y}, \bm{z}; t, 2). \]
\end{theorem}

In other words, \Cref{theo:MapSeries} says that for any partitions $\mu,\nu,\tau \vdash n$ the coefficient $h_{\mu, \nu}^\tau(0) = m_{\mu, \nu}^\tau$ counts some rooted, orientable, bipartite maps and $h_{\mu, \nu}^\tau(1) = \widetilde{m}_{\mu, \nu}^\tau$ counts some rooted, general (i.e. orientable or not), bipartite maps.
Thus we may wonder whether, for a general $\beta := \a - 1$,
the quantity $h_{\mu, \nu}^\tau(\beta)$ also admits 
a nice combinatorial description.
Note that $h_{\mu, \nu}^\tau(\beta)$ is \emph{a priori} a quantity
depending on a parameter $\a$, and describing it as a quantity depending on a different parameter $\beta := \alpha-1$ might seem be artificial.
However, it turned out that this shift seems to be a right one for finding a combinatorial interpretation of $h_{\mu, \nu}^\tau(\beta)$, as suggested by Goulden and Jackson \cite{GouldenJackson1996} in the following conjecture.

\begin{conjecture}[$b$-conjecture]
\label{conj:BConj}
For all partitions $\tau, \mu, \nu \vdash n \geq 1$ the quantity $h_{\mu, \nu}^\tau(\beta)$ can be expressed as:
\begin{equation}
\label{eq:StatisticOnMaps}
h_{\mu, \nu}^\tau(\beta) = \sum_{M}\beta^{\eta(M)},
\end{equation}
where the summation index runs over all rooted, bipartite maps $M$
with the face
distribution $\tau$, the black vertex distribution $\mu$, the white
vertex distribution $\nu$, and where $\eta(M)$ is a
nonnegative integer which is equal to $0$ if and only if $M$ is orientable.
\end{conjecture}

\subsection{\texorpdfstring{$b$}{b}-conjecture and main result}
\label{subsec:Introduction-BConjecture}

The above conjecture is still to be resolved, but some progress towards determining a suitable statistic $\eta$, based on the combinatorial
interpretation of the so-called \emph{marginal sums for maps}, has
been made in the last two decades. Note that there is a natural
bijection between the set of rooted maps with $n$ edges, which are not
necessarily bipartite, and the set of rooted, bipartite maps with the white vertex distribution given by $\nu = (2^n)$. In particular, the following m.g.s 
\[ \Psi(\bm{x}, \bm{y}; t, 1+\beta) := \sum_{n \geq 1} t^n \sum_{\mu, \tau \vdash 2n}h_{\mu, (2^n)}^\tau(\beta)p_\tau(\bm{x}) p_\mu(\bm{y}) \]
is of special interest, as $\Psi(\bm{x}, \bm{y}; t, 1)$ is the
m.g.s. for rooted orientable maps, and $\Psi(\bm{x}, \bm{y}; t, 2)$ is
the m.g.s. for all rooted maps. A formula for $\Psi(\bm{x}, \bm{y}; t,
1+\beta)$ involving the Selberg integral was found by Goulden, Harer and Jackson \cite{GouldenHarerJackson2001}, who suggested that using their formula it is possible to find a combinatorial interpretation of the following marginal sum
\[ l^r_\mu(\beta) = \sum_{\ell(\tau) = r}h_{\mu, (2^n)}^\tau(\beta)\]
in terms of a map statistic as in \cref{eq:StatisticOnMaps} (here $\mu,
\tau \vdash 2n$, and the summation is taken over all partitions $\tau$
which have precisely $r$ nonnegative parts; see
\cref{SubsecPartitions} for a precise definition). A weaker result was first established by Brown and Jackson \cite{BrownJackson2007}, who found some statistics $\eta$ of maps that describe the total marginal sum
\[ k_\mu(\beta) = \sum_{r \ge 1} l^r_\mu(\beta) = \sum_{\tau \vdash 2n}h_{\mu, (2^n)}^\tau(\beta). \]
A simpler description of $\eta$ was found by La Croix \cite{LaCroix2009}, who used it to give a combinatorial description of $l^r_\mu(\beta)$, as suggested by Goulden, Harer and Jackson.

However, not much was known about an algebraic or combinatorial structure of $h_{\mu, \nu}^\tau(\beta)$ for arbitrary partitions $\tau, \mu, \nu \vdash n$, until very recently we have proved in a joint paper with F\'eray \cite{DolegaFeray2016} the following theorem:

\begin{theorem}
\label{theo:Polynomiality}
For all partitions $\tau, \mu, \nu \vdash n \geq 1$ the quantity $h_{\mu, \nu}^\tau(\beta)$ is a polynomial in $\beta$ of degree $2+n-\ell(\tau) - \ell(\mu) - \ell(\nu)$ with rational coefficients.
\end{theorem}

\medskip

In this paper we are focused on a combinatorial part of
$b$-conjecture, especially in the case of rooted, bipartite maps, with
only one face (called \emph{unicellular}). Let us fix a positive
integer $n$, and two partitions $\mu,\nu \vdash n$. According to
\cref{conj:BConj}, there exists some statistic $\eta$ on the set of
all rooted bipartite maps, such that the quantity
$h_{\mu,\nu}^{(n)}(\beta)$ is given by the m.g.s. of rooted,
unicellular, bipartite maps with the black (white, respectively) vertex distribution $\mu$ ($\nu$, respectively). We show that except two special values of $\beta = 0,1$ for which the combinatorial interpretation of $h_{\mu,\nu}^{(n)}(\beta)$ was known, there exists a third special value $\beta = -1$ for which we provide a combinatorial interpretation of $h_{\mu,\nu}^{(n)}(\beta)$. As a result we prove the following:

\begin{theorem}
\label{theo:UnhandledOnefaceMaps}
For all partitions $\mu, \nu \vdash n \geq 1$ 
\[h_{\mu, \nu}^{(n)}(\beta) = \sum_{\mathcal{M}}\beta^{\eta(\mathcal{M})},\]
holds true for $\beta \in \{-1,0,1\}$,
where the summation index runs over all rooted, bipartite unicellular
maps $\mathcal{M}$ with the black vertex distribution $\mu$, the white
vertex distribution $\nu$, and where $\eta(\mathcal{M})$ is a
nonnegative integer which is equal to $0$ if and only if $\mathcal{M}$ is orientable.
\end{theorem}

We show that the top-degree part of the polynomial $h_{\mu,
  \nu}^{(n)}(\beta)$ is equal, up to a sign, to its evaluation at
$\beta = -1$, thus we show that it is given by some rooted, bipartite,
unicellular maps with the black (white, respectively) vertex distribution
$\mu$ ($\nu$, respectively), which are called ``unhandled'' (the
origin of this terminology will be clear later, after we define an
appropriate statistic $\eta$; see \cref{sec:nonorientability}). We
also show that these maps are in a bijection with rooted,
\emph{orientable}, bipartite maps with the black (white, respectively)
vertex distribution $\mu$ ($\nu$, respectively) and with the arbitrary face distribution. 

\medskip

Finally, we show that $b$-conjecture holds true for an infinite family of rooted, unicellular bipartite maps of genus at most $2$:

\begin{theorem}
\label{theo:LowGenera}
For all partitions $\mu, \nu \vdash n \geq 1$ satisfying $\ell(\mu) +
\ell(\nu) \geq n-3$ and $\tau = (n)$ the $b$-conjecture holds true, i.e. 
\[ h_{\mu,\nu}^{(n)}(\beta) = \sum_{\mathcal{M}}\beta^{\eta(\mathcal{M})},\]
where the summation index runs over all rooted, bipartite unicellular
maps $\mathcal{M}$ with the black vertex distribution $\mu$, the white vertex distribution $\nu$, and $\eta(\mathcal{M})$ is a
nonnegative integer which is equal to $0$ if and only if $\mathcal{M}$ is orientable.
\end{theorem}

\subsection{Related problems}
\label{subsec:related}

% We finish this section by mentioning two very similar problems.
% First, a very similar conjecture to \cref{conj:BConj} was also stated
% by Goulden and Jackson \cite{GouldenJackson1996} for power-sum
% expansion of the m.g.s that is inside the paranthesis of $\log$ in \cref{EqDefH}.
% The latter is conjecturally a m.g.s. of \emph{matchings},
% where the exponent of $\beta$ is some combinatorial integer-valued statistic.
% Some special cases of the conjecture have been solved by Goulden and Jackson in their original article
% \cite{GouldenJackson1996} and recently by Kanunnikov and Vassilieva
% \cite{KanunnikovVassilieva2016, Vassilieva2014}, but in general the
% conjecture is still open.
% The polynomiality of the coefficients of the studied m.g.s. was proven by the author of this paper together with F\'eray \cite{DolegaFeray2014}.
% \medskip

A second related problem is the investigation of \emph{Jack
  characters} -- suitably normalized coefficients of the power-sum
symmetric function expansion of Jack polynomials. It was suggested by
Lassalle that a combinatorial description of these objects might
exist. This combinatorial setup was indicated by some polynomiality
and positivity conjectures that he stated
in a series of papers \cite{Lassalle2008a, Lassalle2009}.
Although these conjectures are not fully resolved, it was proven by us
together with \'Sniady \cite{DolegaFeraySniady2014} that in some special
cases bipartite maps together with some statistics that ``measures
their non-orientability'' give the desired combinatorial setup.
Even more, \'Sniady \cite{Sniady2015} found the top-degree part of the
Jack character indexed by a single partition with respect to some gradation. His result states that this top-degree part can be written
as a linear combination of certain functionals, where the index set is
the set of rooted, \emph{orientable}, bipartite maps with the
arbitrary face distribution. While conjecturally it should be
expressed as a linear combination of the same functionals, where the
index set is a set of some special rooted, unicellular, bipartite
maps. \'Sniady was able to find a bijection between these two index sets
\cite{SniadyPrivate}, which inspired us to investigate the
combinatorial side of $b$-conjecture in the case of unicellular maps,
presented in this paper.

\emph{Note added in revision:} After submission of the current paper,
the aforementioned result of \'Sniady appeared in a joint paper with Czy\.zewska-Jankowska \cite{CzyzewskaJankowskaSniady2016}.
\medskip

We cannot resist stating that there must be a deep connection between all these problems, and understanding it would be of great interest.

\subsection{Organization of the paper}
In \cref{sec:preliminaries} we describe all necessary definitions and
background. Then, we introduce a family of statistics of the maps and
we study their properties in
\cref{sec:nonorientability}. \cref{sec:unicellular} is devoted to the
proof of \cref{theo:UnhandledOnefaceMaps} and its consequence, which
says that the family of statistics presented in the previous section describes the top-degree part of the polynomial $h_{\mu\nu}^{(n)}(\beta)$ associated with unicellular maps. In \cref{sec:lowGenera} we introduce some special subfamily of the statistics presented in \cref{sec:nonorientability}, we study their properties and we give a proof of \cref{theo:LowGenera}. We finish this paper by stating some concluding remarks and questions in \cref{sec:conclude}.

\section{Preliminaries}
\label{sec:preliminaries}

  \subsection{Partitions}
  \label{SubsecPartitions}
  We call $\lambda := (\lambda_1, \lambda_2, \dots, \lambda_l)$ \emph{a partition} of $n$
  if it is a weakly decreasing sequence of positive
  integers such that $\la_1+\la_2+\cdots+\la_l = n$.
  Then $n$ is called {\em the size} of $\lambda$ while $l$ is {\em its length}.
  As usual we use the notation $\la \vdash n$, or $|\la| = n$, and $\ell(\la) = l$.
  We denote the set of partitions of $n$ by $\Y_n$ and we define a partial order on $\Y_n$,
called the {\em dominance order}, in the following way:
\[ \lambda \leq \mu \iff \sum_{i\leq j}\lambda_i \leq \sum_{i\leq j}\mu_i \text{ for any positive integer } j.\]
Given two partitions $\la \in \Y_n$ and $\mu \in \Y_m$ we can construct a new partition $\la \cup \mu \in \Y_{n+m}$ obtained by merging parts of $\la$ and $\mu$ and ordering them in a decreasing fashion.

\subsection{Jack polynomials}
\label{subsec:Jack}

In this section we recall the definition of Jack polynomials
and present several known results about them.
Since they are well-established (mostly in a seminal work of Stanley \cite{Stanley1989}),
we do not give any proof, but explicit references.
\bigskip

Consider the vector space $\Symm$ of the symmetric functions $\Lambda$
over the field of rational functions $\QQ(\alpha)$
and endow it with a scalar product $\langle \cdot, \cdot \rangle_{\a}$ defined on the power-sum symmetric functions basis
by the following formula (and then extended by bilinearity):
\[\langle p_\la, p_\mu \rangle_{\a} = z_\la \a^{\ell(\la)} \delta_{\la,\mu},\]
where 
\begin{equation}
\label{eq:NumericalFactor}
z_\la := \prod_{i \geq 1}i^{m_i(\la)}m_i(\la)!.
\end{equation}
Here, $m_i(\la)$ denotes the number of parts of $\la$ equal to $i$.
This is a classical deformation of the \emph{Hall scalar-product} (which corresponds to $\alpha=1$).

Now, Jack polynomials $J_\la^{(\a)}$ are symmetric functions with an additional parameter $\a$ uniquely determined 
(see \cite[Section VI,10]{Macdonald1995}) by the following conditions:
\begin{enumerate}
\item[(C1)]
$J_\la^{(\a)} = \sum_{\mu \leq \la}a^{\la}_\mu m_\mu,$ where
$\a^{\la}_\mu \in \QQ(\a)$;
\item[(C2)]
$[m_{1^{|\la|}}]J_\la^{(\a)} := \a^{\la}_{1^{|\la|}}= |\la|!$;
\item[(C3)]
$\langle J_\la^{(\a)}, J_\mu^{(\a)}\rangle_{\a} = 0$ for $\la \neq \mu$;
\end{enumerate}
where $m_\la$ denotes the monomial symmetric function.

\subsubsection{Basic properties}

We present here several well-known identities for Jack polynomials that will be useful for us later.
\begin{align}
\label{eq:ScalarProd}
\langle J_{(n)}, J_{(n)} \rangle_\a &= (1+\a)(1+2\a)\cdots(1+(n-1)\a) \a^n n!,\\
\label{eq:JackWithOnePart}
J^{(\a)}_{(n)} &= \sum_{\mu \vdash n}\frac{n! \a^{n-\ell(\mu)}}{z_\la}p_\mu,\\
\label{eq:EvaluationInOnePart}
J^{(\a)}_\lambda((t,0,0,\dots)) &= 
\begin{cases} (1+\a)(1+2\a)\cdots(1+(n-1)\a)t^n & \text{for } \lambda = (n),\\ 0  & \text{for } \ell(\lambda) > 1.\end{cases}
\end{align}
\cref{eq:ScalarProd} and \cref{eq:JackWithOnePart} are proved in \cite[Proposition 2.2]{Stanley1989}, and \cref{eq:EvaluationInOnePart} is a consequence of the monomial basis expansion given in 
\cite{KnopSahi1997}.

\subsection{Surfaces, graphs, and maps}
\label{subsect:Maps}

A \emph{map} is an embedding of a connected graph
$G$ into a \emph{surface} $S$ (i.e.~compact, connected, $2$-dimensional
manifold) in a way that the connected components of
$S\setminus G$, called \emph{faces}, are simply connected. Our \emph{graphs} may have loops and multiple
edges. Maps are always considered up to homeomorphisms. A map is
\emph{unicellular} if it has a single face. Unicellular maps are also
called \emph{one-face maps}. We will call a map \emph{orientable} if
the underlying surface is orientable; otherwise we will call it
\emph{non-orientable}. In this paper we will be mostly focused on
non-orientable maps.

\begin{figure}
\centering
\subfloat[]{
	\label{subfig:C1}
	\includegraphics[width=0.3\linewidth]{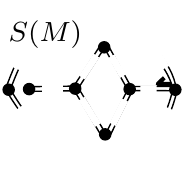}}
\quad
\subfloat[]{
	\label{subfig:M1}
	\includegraphics[width=0.3\linewidth]{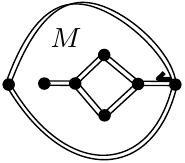}}
\caption{The map $M$ from \cref{fig:MapExampleOrient} is depicted on
  \protect\subref{subfig:M1} as a ribbon
  graph. \protect\subref{subfig:C1} shows stars $S(M)$ associated with $M$.}
\label{fig:CM}
\end{figure}

For the purposes of the current paper it is sometimes convenient to represent 
a map as a \emph{ribbon graph} as follows:
each vertex is represented as a small disc 
and each edge is represented by a thin strip connecting two discs 
in a way that a walk along the boundary of the ribbons corresponds 
to the walk along the boundary of the faces of a given map. To
draw such a picture it is helpful to do it in the following steps:
we first draw vertices (as small discs) together with thin strips attached around
them, which represent associated \emph{half-edges}. By
definition \emph{half-edges} are obtained by removing
middle-points of all the edges (each edge consists of exactly two half-edges). We call this data the \emph{stars of the
map $M$} and we denote it by $S(M)$. See \cref{subfig:C1} for an
example of $S(M)$ associated with the map $M$ from \cref{fig:MapExampleOrient}. Next, for each edge $e$ of $M$, we
can connect the borders of the strips representing two half-edges
belonging to $e$ in two possible ways (see \cref{fig:twist}). We
connect them in a way so that after
connecting all the strips from $S(M)$ the walk along the boundary of the ribbons corresponds 
to the walk along the boundary of the faces of a given
map. \cref{subfig:M1} presents the map $M$ from
\cref{fig:MapExampleOrient} represented as a ribbon graph.

A map is \emph{rooted} if it is equipped with a distinguished half-edge
called the \emph{root}, together with a distinguished side of this half-edge. The vertex incident to the
root is called the \emph{root vertex}, and the edge containing the root is
called the \emph{root edge}. There is an equivalent way to root a map by
choosing a corner (called the \emph{root corner}) and its orientation, where
a \emph{corner} in a map is an angular sector determined by a vertex, and two half-edges which are
consecutive around it. One can then define the root half-edge as the
one lying to the left of the root corner (viewed from the root vertex
and according to the root corner orientation). In this paper we will
use both conventions depending on the situation, and we will represent rooted maps by shading the root corner and/or by indicating the root that is incident to it.

The \emph{degree of a vertex} is the number of incident half-edges, or
equivalently the number of incident corners, while the \emph{degree of
  a face} is the number of corners lying in that face, or equivalently the number of edges incident
to it, with the convention that an edge incident to the same face on both sides
counts for two.

If $M$ is a map, we let $V(M)$, $E(M)$ and $F(M)$ be its
sets of vertices, edges and faces. Their cardinalities $v(M), e(M)$ and $f(M)$ satisfy the \emph{Euler formula}:
\begin{equation}
\label{eq:euler}
e(M) = v(M) + f(M) - 2 + 2g(M),
\end{equation}
where $g(M)$ denotes the \emph{genus} of a map $M$, that is the genus of the underlying surface. We recall that the Euler
characteristic of the surface is $2 - 2g(M)$, thus $g(M)$ is a nonnegative
integer when $M$ is orientable or half-integer when $M$ is
non-orientable. We also denote by $C(M)$ the
set whose elements are indexed by faces of $M$. For a fixed face
$f \in F(M)$ the associated element in $C(M)$ is the set of all
corners belonging to $f$.

\medskip

A map is \emph{bipartite} if its vertices can be colored in two colors in such a way that adjacent
vertices have different colors (say black and white). For a rooted, bipartite map, the color of its root vertex is always taken to be black, by convention.

Let $\mu,\nu,\tau$ be integer partitions. We say that a rooted bipartite map $M$ \emph{has type $(\mu,\nu;\tau)$} if $\mu$ lists the degrees of black vertices (we say \emph{$\mu$ is the black vertex distribution}), $\nu$ lists the degrees of white vertices (\emph{$\nu$ is the white vertex distribution}) and
$\tau$ lists the degrees of faces divided by two (\emph{$\tau$ is the face distribution}). 
We denote the set of rooted bipartite maps of type $(\mu,\nu;\tau)$ on
orientable (all, respectively) surfaces by $\M_{\mu,\nu}^\tau$ 
($\widetilde{\M}_{\mu,\nu}^\tau$, respectively). Note that all three
partitions $\mu,\nu,\tau$ have necessarily the same size $n$, which is
equal to the number of edges of the corresponding map, while its
lengths correspond to the number of black and white vertices and the
number of faces, respectively.

\medskip

\emph{From now on, all the maps are rooted, and bipartite, thus by
  saying a ``map'', what we really mean is a ``rooted, bipartite map''.}

\section{Measure of non-orientability in \texorpdfstring{$b$}{b}--conjecture}
\label{sec:nonorientability}

In this section we are going to construct a function that associates with a map $M$ a nonnegative integer $\eta(M)$ which, in some sense, measures its non-orientability. This ``measure of non-orientability'' gives, in some special cases, a correct answer to $b$-conjecture, i.e.~\cref{eq:StatisticOnMaps} holds true.

The construction presented in this section is due to La Croix \cite{LaCroix2009} who used it to prove that the following marginal sum
\[ l^r_\mu(\beta) = \sum_{\ell(\tau) = r}h_{\mu, (2^n)}^\tau(\beta)\]
can be expressed in the same form as the right hand side of
\cref{eq:StatisticOnMaps}, where the maps in the summation index are
not necessarily bipartite, have $r$ faces and the vertex distribution
$\mu$. The construction of La Croix was originally defined for all
(not necessarily bipartite) maps, but in this paper we are dealing
with the case of bipartite maps, thus in the following all maps will be \emph{rooted and bipartite}.

\subsection{Root-deletion procedure as measure of non-orientability}
\label{subsec:TypesOfEdges}

Let $e$ be the root edge of the map $M$. Note that by deleting $e$ from $M$ we create a new map or, possibly, two new maps and we canonically choose how to root them.
Recall that rooting a map is the same as distinguishing an oriented
corner (called the root corner), see \cref{subsect:Maps}. The root
corner of $M$ is contained in the unique corner $c$ of $M\setminus
\{e\}$ and we set it as the root corner of the connected component of $M\setminus \{e\}$ containing $c$ with an orientation inherited from the root corner of $M$. In the case where deleting $e$ from $M$ decomposes it into two connected components, we additionally distinguish the first corner in the root face of $M$ following the root corner and we notice that it is contained in the unique corner $c'$ of $M\setminus \{e\}$ that belongs to the different connected component of $M\setminus \{e\}$ than the corner $c$. We equip it with the same orientation as the root face of $M$ and we define it as the root corner of the second component of $M\setminus\{e\}$, see \cref{fig:MapWithBridge}.

\begin{figure}
\centering
\includegraphics[width=\linewidth]{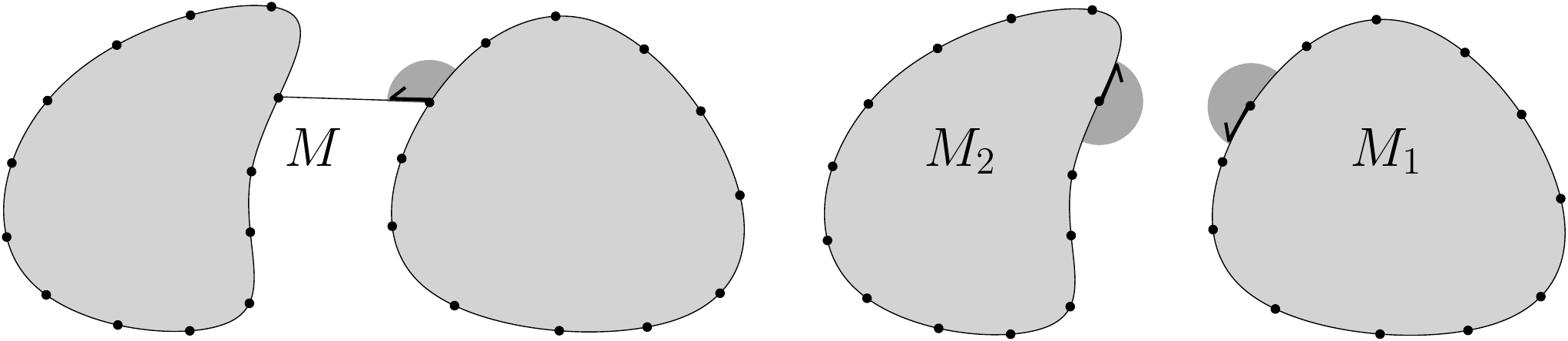}
\caption{A rooted map $M$ on the left hand side and two rooted maps $M_1$ and $M_2$ obtained from $M$ by removing its root edge. Root corners are indicated by dark grey areas.}
\label{fig:MapWithBridge}
\end{figure}

Now, we can classify the root edges $e$ of the map $M$ in the following manner:
\begin{itemize}
\item
if $e$ disconnects $M$ (i.e.~$M\setminus \{e\}$ has two connected
components), $e$ is called a \emph{bridge}; 
\item
otherwise $M\setminus \{e\}$ is connected and there are following possibilities:
\begin{itemize}
\item
the number of faces of $M\setminus \{e\}$ is smaller by $1$ than the
number of faces of $M$ -- in that case $e$ is called a \emph{border};
\item
the number of faces of $M\setminus \{e\}$ is equal to the number of
faces of $M$ -- in that case $e$ is called a \emph{twisted edge};
\item
the number of faces of $M\setminus \{e\}$ is greater by $1$ than the
number of faces of $M$ -- in that case $e$ is called a \emph{handle}.
\end{itemize}
\end{itemize}

\begin{figure}
\centering
\subfloat[]{
	\label{subfig:mapa1}
	\includegraphics[width=0.47\linewidth]{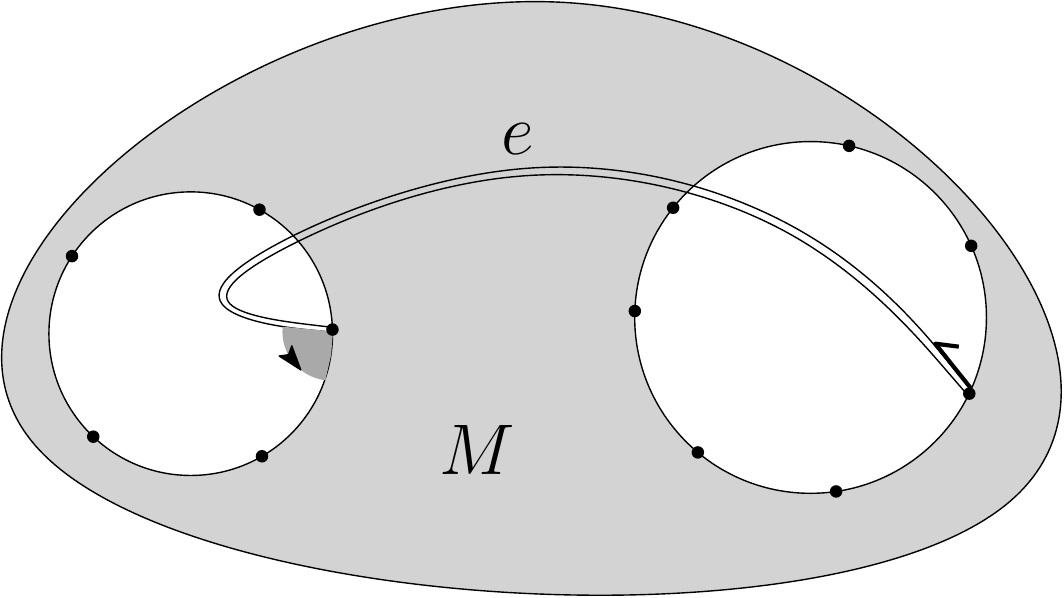}}
\quad
\subfloat[]{
	\label{subfig:mapa2}
	\includegraphics[width=0.47\linewidth]{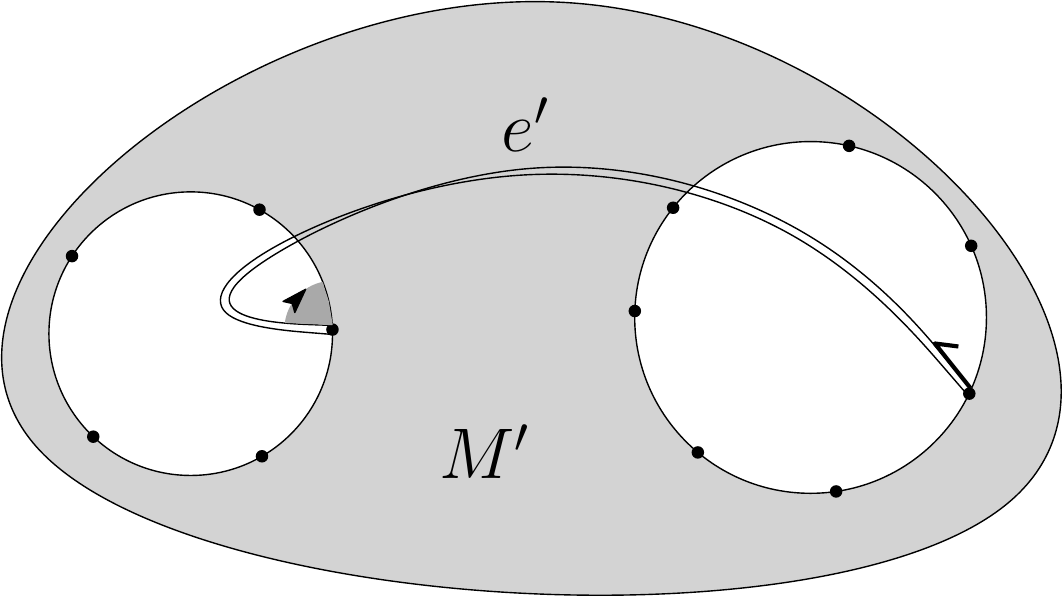}}
\caption{\protect\subref{subfig:mapa1} represents diagrammatically a map $M$, where the root edge $e$ is a handle, and \protect\subref{subfig:mapa2} represents diagrammatically a map $M'$ obtained from $M$ by twisting its root edge $e$, i.e.~the unique map $M'$ different from $M$ such that $M\setminus\{e\} = M'\setminus\{e'\}$, where $e'$ is the root edge of $M'$ and such that two distinct corners of $M\setminus\{e\}$ containing two half-edges of $e$ are the same as   two distinct corners of $M'\setminus\{e'\}$ containing two half-edges of $e'$. Two white areas on \protect\subref{subfig:mapa1} (\protect\subref{subfig:mapa2}, respectively) represents two different faces of $M\setminus\{e\} = M'\setminus\{e'\}$ merged by a root $e$ ($e'$, respectively). To help the reader noticing the difference between $M$ and $M'$, we shade in dark grey and orient according with the root orientation the first visited corner after the root corner in both maps.}
\label{fig:TwistedMaps}
\end{figure}

We are now ready to define a statistic $\eta$ introduced by La Croix.

\begin{definition}
\label{def:MeasureOfNonOrient}{\cite[Definition 4.1]{LaCroix2009}}
A \emph{measure of non-orientability}  is an invariant $\eta(M)$
defined for all rooted maps $M$ such that for any map $M$ the invariant $\eta(M)$ associated
with it satisfies the following properties:
\begin{itemize}
\item
If $M$ has no edges, then $\eta(M) = 0$; 
\item
Otherwise, let $e$ be the root edge of $M$. We have following possibilities:
\begin{itemize}
\item
$e$ is a \emph{bridge}. Then $\eta(M) = \eta(M_1)+\eta(M_2)$, where $M\setminus \{e\} = M_1 \cup M_2$;
\item
$e$ is a \emph{border}. Then $\eta(M) = \eta(M\setminus \{e\})$;
\item
$e$ is \emph{twisted}. Then $\eta(M) = \eta(M\setminus \{e\})+1$;
\item
$e$ is a \emph{handle}. Then there exists a unique map $M'$ with the
root edge $e'$ constructed by twisting the edge $e$ in $M$ such that $e'$ is a handle and such that $M\setminus \{e\} = M'\setminus \{e'\}$. See \cref{fig:TwistedMaps}. In this case we have
\[ \{\eta(M), \eta(M')\} = \{\eta(M\setminus \{e\}), \eta(M\setminus \{e\})+1\}. \]
Moreover, at most one of $M$ and $M'$ is orientable, and its measure of
non-orientability is equal to $0$ while a measure of
non-orientability of the other (nonorientable) map is equal to $1$.
\end{itemize}
\end{itemize}
\end{definition}

\begin{remark}
Note that there are many function $\eta$ satisfying all the conditions given by \cref{def:MeasureOfNonOrient}. Thus, the above definition gives a whole class of functions, and any such function is called a measure of non-orientability.
\end{remark}

Let $M$ be a rooted map. We label all its edges according to their
appearance in the root-deletion procedure. That is, the root edge of
$M$ has label $1$, the root edge of $M\setminus\{e\}$ has label $2$,
etc. Here, there must be a convention chosen in which connected component should be treated first, after removing a bridge. Our convention is that we first decompose the connected component with the root corner that contained the root corner of the previous map. That is, we first decompose the map $M_1$ from \cref{fig:MapWithBridge}.

From now on, we are going to use the following notation: for the rooted
map $M$, and for any $1 \leq i \leq e(M)$ we denote by $e_i(M)$ the
edge with the label $i$ and we set $M^{i+1}$
for the rooted map, which is the connected component of
$M^i\setminus\{e_i(M)\}$ containing $e_{i+1}(M)$. $M^1 := M$ and
$M^{e(M)+1}$ is the unique map with no edges, by convention.

For a given positive integer $n$ and partitions $\mu, \nu, \tau \vdash
n$ we can decompose the set $\widetilde{\mathcal{M}}_{\mu,\nu}^\tau$
of maps of type $(\mu,\nu;\tau)$ in the following manner:
\begin{equation}
\label{eq:HandlesDecomp}
\widetilde{\M}_{\mu,\nu}^\tau = \bigcup_{i \geq 0}\widetilde{\M}_{\mu,\nu;i}^\tau,
\end{equation}
where $\widetilde{\M}_{\mu,\nu;i}^\tau$ is the set of maps of type
$(\mu,\nu;\tau)$ such that exactly $i$ handles appeared during their
root-deletion process. In other words, it is the set of rooted maps
$M$ of type $(\mu,\nu;\tau)$ such that for all natural numbers $k \in \N$, except $i$, the root of $M^k$ is not a handle.
We call maps from the set $\widetilde{\M}_{\mu,\nu;0}^\tau$ \emph{unhandled}.
%We also set 
%\begin{equation}
%\label{eq:MoreHandlesDecomp}
%\widetilde{\M}_{\mu,\nu;\geq i}^\tau, = 
%\bigcup_{j \geq i}\widetilde{\M}_{\mu,\nu;j}^\tau.
%\end{equation}
Finally, we denote the finite set $\{1,2,\dots,n\}$ by $[n]$. Here, we
present a classical, but important for us, relation between a genus of
a given map and its root-deletion procedure.

\begin{lemma}
\label{lem:NrOfHandles}
Let $M \in \widetilde{\M}_{\mu,\nu;i}^\tau$ be a rooted map with $n$ edges, and let $j(M)$ denotes the number of positive integers $j \in [n]$ for which the root of $M^j$ is twisted. Then the following equality holds true:
\begin{equation}
\label{eq:pomocnicze}
j(M) + 2i = 2g(M),    
\end{equation}
where $g(M)$ is a genus of the map $M$.
\end{lemma}

\begin{proof}
We are going to prove that \cref{eq:pomocnicze} holds true for all maps by an induction on the number of edges of $M$.

It is straightforward to check that there is only one rooted, bipartite map with one edge. Its root edge is a bridge and it is planar (i.e. its genus is equal to $0$). Thus, \cref{eq:pomocnicze} holds true in this case.
Now, we fix $n \geq 2$ and we assume that \cref{eq:pomocnicze} holds true for all maps with at most $n-1$ edges. Let $M \in \widetilde{\M}_{\mu,\nu;i}^\tau$ be a map with $n$ edges and let $i(M)$ denotes the number of handles appearing in the root-deletion process of $M$, i.e.~ $i(M) = i$. We are going to analyze how the Euler characteristic varies during the root-deletion process. It is straightforward from the classification of root edges and from \cref{eq:euler} that we have following possibilities:
\begin{itemize}
\item
$e$ is a \emph{bridge}. Then $g(M) = g(M_1)+g(M_2), i(M) = i(M_1) + i(M_2)$, and $j(M) = j(M_1) + j(M_2)$, where $M\setminus \{e\} = M_1 \cup M_2$. Thus, by an inductive hypothesis
\begin{multline*}
j(M) + 2i(M) = j(M_1) + j(M_2) + 2\left(i(M_1) + i(M_2)\right) = 2g(M_1) + 2g(M_2) = 2g(M);    
\end{multline*}
\item
$e$ is a \emph{border}. Then $g(M) = g(M\setminus \{e\}), i(M) = i(M\setminus \{e\})$, and $j(M) = j(M\setminus \{e\})$. Thus, by an inductive hypothesis
\begin{multline*}
j(M) + 2i(M) = j(M\setminus \{e\}) + 2i(M\setminus \{e\}) = 2g(M\setminus \{e\}) = 2g(M);
\end{multline*}
\item
$e$ is \emph{twisted}. Then $g(M) = g(M\setminus \{e\})+1/2, i(M) = i(M\setminus \{e\})$, and $j(M) = j(M\setminus \{e\})+1$. Thus, by an inductive hypothesis
\begin{multline*}
j(M) + 2i(M) = j(M\setminus \{e\}) + 2i(M\setminus \{e\}) + 1= 2g(M\setminus \{e\}) + 1 = 2g(M);
\end{multline*}
\item
$e$ is a \emph{handle}. Then $g(M) = g(M\setminus \{e\})+1, i(M) = i(M\setminus \{e\})+1$, and $j(M) = j(M\setminus \{e\})+1$. Thus, by an inductive hypothesis
\begin{multline*}
j(M) + 2i(M) = j(M\setminus \{e\}) + 2i(M\setminus \{e\}) + 2= 2g(M\setminus \{e\}) + 2 = g(M).
\end{multline*}
\end{itemize}
Since these are all possible cases, we proved by induction that \cref{eq:pomocnicze} holds true for any rooted, bipartite map $M$, which finishes the proof.
\end{proof}

\begin{corollary}
\label{cor:EtaDegree}
Let $M \in \widetilde{\M}_{\mu,\nu;i}^\tau$ be a rooted map with $n$ edges. Then 
\begin{multline*}
0 \leq n+2 - \left(\ell(\mu)+\ell(\nu)+\ell(\tau)\right) - 2i
\leq \eta(M) \leq n+2 - \left(\ell(\mu)+\ell(\nu)+\ell(\tau)\right) - i.
\end{multline*}
\end{corollary}

\begin{proof}
The Euler formula given by \cref{eq:euler} yields
\[2g(M) = 2 + e(M) - f(M) - v(M) = n + 2 - \left(\ell(\mu)+\ell(\nu)+\ell(\tau)\right).\]
Combining it with \cref{eq:pomocnicze} we have the following formula:
\[ 0 \leq j(M) = n + 2 - \left(\ell(\mu)+\ell(\nu)+\ell(\tau)\right) - 2i(M).\]
It is now enough to notice an obvious inequality which comes strictly
from the definition \cref{def:MeasureOfNonOrient} of $\eta$:
\begin{multline*}
0 \leq n + 2 - \left(\ell(\mu)+\ell(\nu)+\ell(\tau)\right) - 2i(M) = j(M) \\
\leq \eta(M) \leq j(M) + i(M) = n + 2 - \left(\ell(\mu)+\ell(\nu)+\ell(\tau)\right) - i(M),    
\end{multline*}
which finishes the proof.
\end{proof}

\subsection{Twist involution}

Let $M$ be a rooted map such that its root edge $e$ is a handle. We
recall that there exists the unique rooted map $M'$ different from
$M$ with the root edge $e'$ which is a handle, too, and such that the
half-edges belonging to $e$ are lying in the same corners of $M
\setminus \{e\} = M' \setminus \{e'\}$, as the half-edges belonging to
$e'$. Notice that the map $M'$ is, roughly speaking, obtained from $M$
by ``twisting'' its root. In this section we are going to formalize
and generalize the concept of ``twisting edges''.

\begin{definition}
\label{def:twist}
Let $M$ be a rooted map with $n$ edges and let us fix an integer $ i
\in [n]$. We denote by $h_i(M)$ the root of $M^i$ and by $h_i(M)'$
the second half-edge belonging to $e_i(M)$. Let $c_1$ and
$\tilde{c}_1$ be two corners adjacent to $h_i(M)$ and oriented towards
$h_i(M)'$. We denote by $\tau_iM$ the map whose ribbon graph is
obtained from the ribbon graph of $M$ by ``twisting'' the edge
$e_i(M)$. That is, by connecting the half-edges $h_i(M)$ and $h_i(M)'$
in the (unique!) different way than they are connected in $M$. One can
describe this construction in a more formal way as follows. Let $c_2$
($\tilde{c}_2$, respectively) be the unique oriented corner adjacent
to $h_i(M)'$, which is the first corner visited after $c_1$
($\tilde{c}_1$, respectively) -- see \cref{subfig:twist1}. There
exists a unique map $\tau_i M$ obtained from $M$ by replacing the
edge $e_i(M)$ by the edge $e_i'$ connecting $h_i(M)$ with $h_i(M)'$ in
$\tau_iM$such that the oriented corner adjacent to $h_i(M)'$ and visited after $c_1$ ($\tilde{c}_1$, respectively) is the corner $\tilde{c}_2$ ($c_2$, respectively) -- see \cref{subfig:twist2}. We call the operator $\tau_i$ \emph{twisting} of $i$-th edge of $M$.
\end{definition}

\begin{remark}
Note that $M$ and $\tau_i M$ are the same graphs (thus the sets
$E(M)$, and $E(\tau_i M)$ are the same, and it makes sense to compare
properties of an edge $e$ in $M$ to its properties in $\tau_i M$), but
it is not true in general that $\tau_i \tau_j M$ is the same map as
$\tau_j \tau_i M$. The following proposition resolves when the
twisting operators commute.
\end{remark}

\begin{figure}
\centering
\subfloat[]{
	\label{subfig:twist1}
	\includegraphics[width=0.47\linewidth]{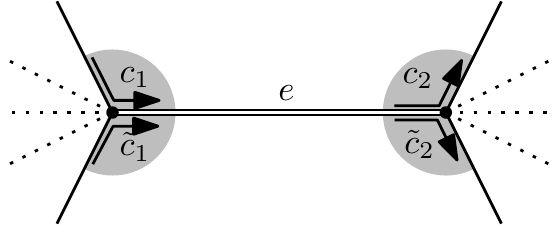}}
\quad
\subfloat[]{
	\label{subfig:twist2}
	\includegraphics[width=0.47\linewidth]{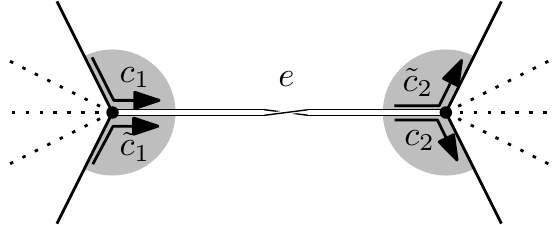}}
\caption{}
\label{fig:twist}
\end{figure}

\begin{proposition}
\label{prop:blabla}
Fix a positive integer $n$, partitions $\mu,\nu \vdash n$, and a map $M$ with $n$ edges. Then
\begin{enumerate}
\item 
for any $i \in [n]$ the operator $\tau_i$ is an involution on the set of maps with black (white, respectively) vertex distribution $\mu$ ($\nu$, respectively),
\item
let $I = \{i_1 < \cdots < i_k\}$ be a non-empty subset of $[n]$ such
that for all $i \in I$ the root edge of $M^i$ is not a bridge. Then, for any permutation $\sigma \in \Sym{I}$ the map $\tau_{\sigma(i_1)}\cdots\tau_{\sigma(i_k)}M$ is the same, the labels of the edges in $M$ and in $\tau_{\sigma(i_1)}\cdots\tau_{\sigma(i_k)}M$ coincide, and for any $j \in [n]$ the root edge of $M^j$ is a bridge iff the root edge of $(\tau_{i_1}\cdots \tau_{i_k} M)^j$ is a bridge.
\end{enumerate}
\end{proposition}

\begin{proof}
Let us fix $i \in [n]$. It is clear from the construction that the
operator $\tau_i$ preserves white and black vertex distributions. Now,
notice that twisting an $i$-th edge in any map $M$ does not change the
labels of the first $i$ edges and it may change the labels of the
other edges only if the edge $e_i(M)$ is a bridge in $M^i$. Thus, the
first $i$ labels of the edges in both maps $M$ and $\tau _i M$ are the
same, so $\tau_i^2 M = M$.

We are going to prove the second item by an induction on size of
the set $I$. We have already proved above that the inductive assertion
holds true for $|I| = 1$, so let us fix an integer $1 < k \leq n$. We assume that the inductive assertion holds true for all subsets $I
\subset [n]$ of size smaller then $k$. Let $I = \{i_1 < \cdots <
i_k\}$ be a non-empty subset of $[n]$ such that for all $i \in I$ the
root edge of $M^i$ is not a bridge and let $\si \in \Sym{I}$.
Then, there exists an integer $l \in [k]$ and a permutation $\pi \in \Sym{[k]\setminus\{l\}}$ such that
\[ \tau_{\sigma(i_1)}\cdots\tau_{\sigma(i_k)}M = \tau_{i_l} \left(
  \tau_{i_{\pi(1)}}\cdots\widehat{\tau_{i_{\pi(l)}}}\cdots\tau_{i_{\pi(k)}}
  M\right),\]
where we use a standard notation that the word $a_1\cdots
a_{i-1}\widehat{a_{i}} a_{i+1}\cdots a_n$ is obtained from the word $a_1\cdots
a_{i-1} a_{i} a_{i+1}\cdots a_n$ by removing the letter $a_i$.
If $l = 1$, then the labels in $M$ and $\left(
  \tau_{i_{\pi(1)}}\cdots\widehat{\tau_{i_{\pi(l)}}}\cdots\tau_{i_{\pi(k)}}
  M\right)$ coincide, and the root edge of $\left(
  \tau_{i_{\pi(1)}}\cdots\widehat{\tau_{i_{\pi(l)}}}\cdots\tau_{i_{\pi(k)}}
  M\right)^{i_1}$ is not a bridge since the root edge of $M^{i_1}$ is
not a bridge, by the inductive assertion. Thus, the labels in $M$, and in $\tau_{\sigma(i_1)}\cdots\tau_{\sigma(i_k)}M$ coincide, too.
If $l > 1$, then the inductive assertion says that
the labels of $M$ and $\tau_{j_1}\cdots\tau_{j_m} M$ are the same for
all subsets $\{j_1,\dots,j_m\} \subset I$ of size $m<k$, so
\begin{multline} 
\label{eq:takietam}
\tau_{i_l}\left(
  \tau_{i_{\pi(1)}}\cdots\widehat{\tau_{i_{\pi(l)}}}\cdots\tau_{i_{\pi(k)}}
  M\right) = \tau_{i_l}\left(\tau_{i_1}\cdots \widehat{\tau_{i_l}}\cdots \tau_{i_k} M\right) \\
= \tau_{i_1}\left(\tau_{i_l}\tau_{i_2}\cdots \widehat{\tau_{i_l}}\cdots \tau_{i_k} M\right) = \tau_{i_1}\cdots \tau_{i_k} M
\end{multline}
by the inductive hypothesis.
Moreover, again by the inductive hypothesis, the labels in the maps were not changed when we were swapped operators $\tau$ with different indices.
Finally, for any $j \in [n]$, the graphs $M^j$ and
$\left(\tau_{i_1}\cdots \tau_{i_k} M\right)^j$ are the same. Since being a bridge is the same as being a disconnecting edge of the graph, the proof is finished.
\end{proof}

\begin{lemma}
\label{lem:involution}
Let $\eta$ be a measure of non-orientability. Then, for all positive integers $n,i$ and partitions $\mu,\nu \vdash n$, there exists an involution $\si_\eta$ on the set $\widetilde{\M}_{\mu,\nu;i}^{(n)}$, which has the property that 
\[(-1)^{\eta(\si_\eta(M))} = (-1)^{\eta(M)+1}.\]
Moreover, for each $M \in \widetilde{\M}^{(n)}_{\mu,\nu; i}$ there exist natural numbers $1 \leq i_1 < \cdots < i_k \leq n$ such that $\si_\eta(M) = \tau_{i_k}\cdots\tau_{i_1}M$ and such that for each $j \in [k]$ the root of $M^{i_j}$ is not a bridge.
\end{lemma}

\begin{proof}
We are going to construct $\si_\eta$ by induction on $n$. For $n = 1$ all rooted bipartite maps with $n$ edges are unhandled so we set $\si_\eta$ as an empty map. 

We denote by $i(M)$ the number of handles appearing in the
root-deletion process of $M$. We fix $n \geq 1$ and we assume that the
involution $\si_\eta$ is already defined for all unicellular maps with
at most $n$ edges which are not unhandled. Let $M$ be a unicellular
map with $n+1$ edges. Since $M$ has a unique face, $e_1(M)$ cannot be
a border so there are the following possibilities:
\begin{itemize}
\item
$e_1(M)$ is a handle. In this case we set $\si_\eta(M) = \tau_1 M$. Clearly $(-1)^{\eta(M)} = (-1)^{\eta(\si_\eta(M))+1}$ and $i(M) = i(\si_\eta(M))$;
\item
$e_1(M)$ is twisted. In this case $M^2$ is a unicellular map with
$n$ edges. Thus, by the inductive hypothesis there exist natural numbers
$1 \leq i_1 < \cdots < i_k \leq n$ such that $i(M^2) = i(\tau_{i_k}
\cdots \tau_{i_1}M^2)$ and such that 
\[(-1)^{\eta(\tau_{i_k}\cdots\tau_{i_1}M^2)} = (-1)^{\eta(M^2)+1}.\] 

Let us consider two rooted maps $M_1 = \tau_{i_k+1}\cdots\tau_{i_1+1}M$ and $M_2 = \tau_{i_k+1}\cdots\tau_{i_1+1} \tau_1 M$. Both maps $M_1$ and $M_2$ have a property that after deletion of its root we obtain the same rooted map $M_1^2 = M_2^2 = \tau_{i_k}\cdots\tau_{i_1}M^2$ (\cref{prop:blabla} asserts that the labels in $M_1^2$ and in $M_2^2$, respectively, correspond to the labels of $M_1$ and $M_2$, respectively, shifted by $1$). Thus, exactly one map from $M_1$ and $M_2$ is a map $M'$ with the unique face (and its root is twisted) while the second one has two faces (and its root is a border) and we set $\si_\eta(M) = M'$. Strictly from the construction, one has
\[(-1)^{\eta(\si_\eta(M))} = (-1)^{\eta(\si_\eta(M^2))+1} =
(-1)^{\eta(M^2)} = (-1)^{\eta(M) + 1} \]
and 
\[i(M) = i(M^2) = i(\si_\eta(M^2)) = i(\si_\eta(M)^2) = i(\si_\eta(M));\]
\item
$e_1(M)$ is a bridge. In this case $M^2 = M_1 \cup M_2$ is a disjoint
sum of two unicellular maps (and we recall the convention for
labeling: the edge $e_2(M)$ belongs to $M_1$). If $M_1$ is not
unhandled, then there exists a positive integer $k$ and natural
numbers $1 \leq i_1 < \cdots < i_k \leq e(M_1)$ such that
$i(\tau_{i_k}\cdots\tau_{i_1}M_1) = i(M_1)$ and such that 
\[(-1)^{\eta(\tau_{i_k}\cdots\tau_{i_1}M_1)} = (-1)^{\eta(M_1)+1}.\] 
In this case we set $\si_\eta(M) := \tau_{i_k+1}\cdots\tau_{i_1+1}M$
and since $\si_\eta(M)^2 = \si_\eta(M_1) \cup M_2$ it
is clear that $(-1)^{\eta(\si_\eta(M))} = (-1)^{\eta(M)+1}$ and
\[i(M) = i(M_1)+i(M_2) = i(\si_\eta(M_1)) + i(M_2) = i(\si_\eta(M)).\] 
If the map $M_1$ is unhandled, then the map $M_2$ is not unhandled and there exist a positive integer $k$ and natural numbers $1 \leq i_1 < \cdots < i_k \leq e(M_2)$ such that $i(\tau_{i_k}\cdots\tau_{i_1}M_1) = i(M_1)$ and such that 
\[(-1)^{\eta(\tau_{i_k}\cdots\tau_{i_1}M_2)} = (-1)^{\eta(M_2)+1}.\] 
In this case we set $\si_\eta(M) :=
\tau_{i_k+e(M_1)+1}\cdots\tau_{i_1+e(M_1)+1}M$. Since
$\left(\si_\eta(M)\right)^2 = M_1 \cup \si_\eta(M_2)$ it is clear
that $(-1)^{\eta(\si_\eta(M))} = (-1)^{\eta(M)+1}$ and 
\[i(M) = i(M_2) = i(\si_\eta(M_2)) = i(\si_\eta(M)).\] 
\end{itemize}
Now it is straightforward from the construction and from the
inductive hypothesis that if $\si_\eta(M)$ associated with the rooted
map $M$ is of the form $\tau_{i_k}\cdots\tau_{i_1}M$, then
$\si_\eta(\si_\eta(M))$ is of the same form,
i.e. $\si_\eta(\si_\eta(M)) =
\tau_{i_k}\cdots\tau_{i_1}\si_\eta(M)$. But for each $j \in [k]$ the
root edge of $M^{i_j}$ is not a bridge. Thus
\[ \si_\eta(\si_\eta(M)) = \tau_{i_k}\cdots\tau_{i_1}(\tau_{i_k}\cdots\tau_{i_1}M) = \tau_{i_k}\cdots\tau_{i_1}(\tau_{i_1}\cdots\tau_{i_k}M) = M,\]
where the last equalities come from \cref{prop:blabla}, which finishes the proof.
\end{proof}

\subsection{Algebraic properties of a measure of non-orientability}

Let $\eta$ be a measure of non-orientability and let $\mu,\nu,\tau$
be partitions of a positive integer $n$. We define the following statistic associated with $\eta$:
\begin{equation}
\label{eq:LaCroixMeasure}
\left(H_\eta\right)_{\mu, \nu}^\tau(\beta) := \sum_{M \in \widetilde{\M}^\tau_{\mu,\nu}}\beta^{\eta(M)}.
\end{equation}

The main purpose of this section is to investigate algebraic properties of $\left(H_\eta\right)_{\mu,\nu}^\tau$ that will be of the great importance in the proof of \cref{theo:UnhandledOnefaceMaps}. From now on, we fix a positive integer $n$, partitions $\mu,\nu,\tau \vdash n$, and a measure of non-orientability $\eta$.

\begin{proposition}
\label{prop:a-ki}
Let $g := n + 2 - \left(\ell(\mu) + \ell(\nu) + \ell(\tau)\right)$. Then, for any nonnegative integer $i$, the following quantity
\begin{equation}
\label{eq:a-ki}
\left(a_\eta\right)_{\mu,\nu;i}^\tau(\beta) := \sum_{M \in \widetilde{\M}^\tau_{\mu,\nu;i}}\beta^{\eta(M)+2i-g}
\end{equation}
is a polynomial in $\beta$ of degree at most $i$.
\end{proposition}

\begin{proof}
It is a direct consequence of \cref{cor:EtaDegree}, which says that for any map $M \in \widetilde{\M}^\tau_{\mu,\nu;i}$ the following inequalities hold: 
\[ 0 \leq \eta(M) + 2i - g \leq i.\]
\end{proof}

\begin{corollary}
The quantity $\left(H_\eta\right)_{\mu, \nu}^\tau(\beta)$ is a polynomial in $\beta$ with positive integer coefficients. Moreover, it has the following form:
\begin{equation}
\label{eq:HPolynomialForm}    
\left(H_\eta\right)_{\mu, \nu}^\tau(\beta) = \sum_{0 \leq i \leq [g/2]} \left(a_\eta\right)_{\mu,\nu;i}^\tau(\beta) \beta^{g-2i},
\end{equation}
where $g := n + 2 - \left(\ell(\mu) + \ell(\nu) + \ell(\tau)\right)$.
\end{corollary}

\begin{proof}
Strictly from the definition of $\left(H_\eta\right)_{\mu, \nu}^\tau(\beta)$ given by \cref{eq:LaCroixMeasure}, one has the following formula:
\[ \left(H_\eta\right)_{\mu, \nu}^\tau(\beta) = \sum_{i \geq 0}\sum_{M \in \widetilde{\M}^\tau_{\mu,\nu;i}}\beta^{\eta(M)} = \sum_{0 \leq i \leq [g/2]} \left(a_\eta\right)_{\mu,\nu;i}^\tau(\beta) \beta^{g-2i},\]
where the last equality is simply a definition of $a_{\mu,\nu;i}^\tau(\beta)$ given by \cref{eq:a-ki}.
\end{proof}

\begin{proposition}
\label{lem:LaCroixPolynomials}
For any positive integer $i \geq 1$ one has
\[ \left(a_\eta\right)_{\mu,\nu;i}^{(n)}(-1) = 0.\]
\end{proposition}

\begin{proof}
Plugging $\beta = -1$ into \cref{eq:a-ki} one has
\[(-1)^{n+1 - \ell(\mu) - \ell(\nu)}\left(a_\eta\right)_{\mu,\nu;i}^{(n)}(-1) = \sum_{M \in \widetilde{\M}^{(n)}_{\mu,\nu;i}}(-1)^{\eta(M)}.\]
\cref{lem:involution} says that for each $i \geq 1$ there exists an involution $\sigma_\eta$ on the set $\widetilde{\M}^{(n)}_{\mu,\nu;i}$ which has a property that $(-1)^{\eta(\sigma_\eta(M))} = (-1)^{\eta(M)+1}$. This means that
\begin{multline*}
\sum_{M \in \widetilde{\M}^{(n)}_{\mu,\nu;i}}(-1)^{\eta(M)} =   \sum_{M \in \widetilde{\M}^{(n)}_{\mu,\nu;i}}(-1)^{\eta(\sigma_\eta(M))} = -\sum_{M \in \widetilde{\M}^{(n)}_{\mu,\nu;i}}(-1)^{\eta(M)} = 0.
\end{multline*}
Thus, $\left(a_\eta\right)_{\mu,\nu;i}^{(n)}(-1) = 0$ which finishes the proof.
\end{proof}

\begin{corollary}
The following equality holds true:
\begin{equation} 
\label{eq:Basis}
(-1)^{n+1 - \ell(\mu) - \ell(\nu)}\left(H_\eta\right)_{\mu, \nu}^{(n)}(-1) = \left(a_\eta\right)_{\mu,\nu;0}^{(n)}(-1) = \#\widetilde{\M}^{(n)}_{\mu,\nu;0}.
\end{equation}
\end{corollary}

\begin{proof}
It is enough to plug $\beta = -1$ into \cref{eq:HPolynomialForm} to obtain
\[(-1)^g\left(H_\eta\right)_{\mu, \nu}^{(n)}(-1)= \sum_{i \geq 0}\left(a_\eta\right)_{\mu,\nu;i}^{(n)}(-1) = \left(a_\eta\right)_{\mu,\nu;0}^{(n)}(-1),\]
where $g := n+1 - \ell(\mu) - \ell(\nu)$ and the last equality is a consequence of \cref{lem:LaCroixPolynomials}.
An equality 
\[ \left(a_\eta\right)_{\mu,\nu;0}^{(n)}(-1) = \#\widetilde{\M}^{(n)}_{\mu,\nu;0}\]
is obvious from \cref{cor:EtaDegree}.
\end{proof}

\section{\texorpdfstring{$b$}{b}--conjecture for unicellular maps and measure of non-orientability}
\label{sec:unicellular}
\subsection{Marginal sum}

We are going to prove that fixing white and black vertex distributions and allowing any face distribution,
the corresponding sum of coefficients in $\psi(\bm{x}, \bm{y}, \bm{z}; t, 1+\beta)$ is given by a measure of non-orientability of the appropriate maps. 
The developments in this section are similar to that of \cite[Section 3.5]{BrownJackson2007}
except that here we work in a more general setup (in \cite[Section
3.5]{BrownJackson2007} $\nu=(2^{n/2})$ with $n$ even, while here $\nu$
is an arbitrary partition) and with a slightly different function $\eta$.
We start with the following proposition.

\begin{proposition}
\label{prop:marginalsum}
For any positive integer $n$ and for any partitions $\mu,\nu \vdash n$, the following identity holds true:
\[ \sum_{\tau \vdash n} h_{\mu,\nu}^\tau(\beta) = (1+\beta)^{n+1-\ell(\mu)-\ell(\nu)}\sum_{\tau \vdash n} h_{\mu,\nu}^\tau(0).\]
\end{proposition}

\begin{proof}
We know that
\[ \sum_{\tau \vdash n} h_{\mu,\nu}^\tau(\beta) = [t^n p_\mu(\bm{x})p_\nu(\bm{y})]\psi(\bm{x}, \bm{y}, \bm{z}; t, \a)\bigg|_{\bm{z}=(1,0,0\dots)} \]
because of the trivial identity
\[ J^{(\a)}_\lambda(1,0,0,\dots) = J^{(\a)}_\lambda(\bm{x})\bigg|_{p_1(\bm{x}) = p_2(\bm{x}) = \cdots = 1}.\]
Using \cref{eq:EvaluationInOnePart} and replacing the scalar product by its expression given in \cref{eq:ScalarProd}, we obtain
\begin{equation} 
\sum_{\tau \vdash n} h_{\mu,\nu}^\tau(\beta) = 
[t^n p_\mu(\bm{x})p_\nu(\bm{y})]
(1+\beta)t\frac{\partial}{\partial_t}\log\left(\sum_{n \geq 0}t^n\frac{J^{(\a)}_{(n)}(\bm{x})J^{(\a)}_{(n)}(\bm{y})}{\a^n n!}\right).
\end{equation}
The formula for Jack polynomials indexed by one-part partitions given
in \cref{eq:JackWithOnePart} leads to the following equality:
\begin{multline*} 
\sum_{\tau \vdash n} h_{\mu,\nu}^\tau(\beta) = (1+\beta)[t^n p_\mu(\bm{x})p_\nu(\bm{y})] t\frac{\partial}{\partial_t}\log\left(\sum_{n \geq 0}t^n\sum_{\la^1, \la^2 \vdash n}\a^{n-\ell(\la^1)-\ell(\la^2)} \frac{n!p_{\la^1}(\bm{x})p_{\la^2}(\bm{y})}{z_{\la^1}z_{\la^2}}\right) \\ =
(1+\beta)^{n+1-\ell(\mu)-\ell(\nu)}[t^n p_\mu(\bm{x})p_\nu(\bm{y})]t\frac{\partial}{\partial_t}\log\left(\sum_{n \geq 0}t^n\sum_{\la^1, \la^2 \vdash n} \frac{n!p_{\la^1}(\bm{x})p_{\la^2}(\bm{y})}{z_{\la^1}z_{\la^2}}\right).
\end{multline*}
But the last expression is simply equal to
\begin{multline*} 
(1+\beta)^{n+1-\ell(\mu)-\ell(\nu)}[t^n p_\mu\big(\bm{x})p_\nu(\bm{y})]\ \psi(\bm{x}, \bm{y}, \bm{z}=(1,0,0\dots); t, 1\big) \\
(1+\beta)^{n+1-\ell(\mu)-\ell(\nu)}\sum_{\tau \vdash n} h_{\mu,\nu}^\tau(0),
\end{multline*}
which finishes the proof.
\end{proof}

We can prove now that the following marginal sum is given by a measure of non-orientability:

\begin{theorem}
\label{ptheo:marginalSumEta}
For any measure of non-orientability $\eta$, for any positive integer $n$, and for any partitions $\mu, \nu \vdash n$, the following identity holds true:
\begin{equation} 
\label{eq:marginalSumEta}
\sum_{\tau \vdash n}h_{\mu,\nu}^{\tau}(\beta) = \sum_{\tau \vdash n}\left(H_\eta\right)_{\mu,\nu}^{\tau}(\beta),
\end{equation}
where $\left(H_\eta\right)_{\mu,\nu}^{\tau}(\beta)$ are given by \cref{eq:LaCroixMeasure}.
\end{theorem}

\begin{proof}
We recall that $\left(H_\eta\right)_{\mu,\nu}^{\tau}(\beta)$ is defined as a weighted sum of some rooted, bipartite maps (see \cref{eq:LaCroixMeasure}). Thus, one can define a statistic $\left(H_\eta\right)_{\mu,\nu;i}(\beta)$ as the right hand side of \eqref{eq:marginalSumEta} with a summation restricted to the maps with the root vertex of degree $i$. We are going to prove a stronger result, namely
\begin{equation} 
\label{eq:TraceRoot}
\left(H_\eta\right)_{\mu,\nu;i}(\beta) = (1+\beta)^{n+1-\ell(\mu)-\ell(\nu)}\tilde{H}_{\mu,\nu;i},
\end{equation}
where $\tilde{H}_{\mu,\nu;i}$ is the number of \emph{orientable} maps
with the root vertex of degree $i$, the black vertex distribution
$\mu$, and the white vertex distribution $\nu$. 
If \cref{eq:TraceRoot} holds true, then
\begin{multline*}
    \sum_{\tau \vdash n}\left(H_\eta\right)_{\mu,\nu}^{\tau}(\beta)
= \sum_{i \geq 1}\left(H_\eta\right)_{\mu,\nu;i}(\beta) = (1+\beta)^{n+1-\ell(\mu)-\ell(\nu)}\sum_{i \geq 1}\tilde{H}_{\mu,\nu;i} 
\\ = (1+\beta)^{n+1-\ell(\mu)-\ell(\nu)}\sum_{\tau \vdash n} h_{\mu,\nu}^\tau(0) = \sum_{\tau \vdash n}h_{\mu,\nu}^{\tau}(\beta).
\end{multline*}
The third equality uses the combinatorial interpretation of $h_{\mu,\nu}^\tau(0)$ (see \cref{theo:MapSeries})
while the last equality comes from \cref{prop:marginalsum}.
In this way we have shown that \cref{eq:TraceRoot} implies \cref{eq:marginalSumEta}. Thus, it is sufficient to prove \cref{eq:TraceRoot}.

Before we start a proof we introduce some notation. Let
$r_1,\dots,r_k$ be some positive integers such that $r_1 + \cdots +
r_k = n$. Let us fix a partition $\mu \vdash n$. We define
$\Sp^{(r_1,\dots,r_k)}(\mu)$ as the set of sequences of partitions
$(\mu^1,\dots,\mu^k)$ such that $\mu^1 \vdash r_1, \dots, \mu^k \vdash
r_k$ and such that their sum gives the fixed partition $\mu$. That is,
$\bigcup_{1 \leq i \leq k}\mu^i = \mu$. Moreover, for any positive
integer $i \geq 1$ and partition $\mu$ containing a part equal to $i$, we set
\[\mu_{\downarrow (i)} = \big( \mu \setminus (i)\big) \cup (i-1).\]

We are going to prove \cref{eq:TraceRoot} by induction on $n$.
Let $n = 1$; there exists only one partition of size $n$. That is,
$\mu = \nu = (1)$. Moreover, there is only one map with one edge, and it is planar, so clearly $H_{(1), (1);1}(\beta) = 1 = \tilde{H}_{(1), (1);1}$.
Let us fix $n \geq 2$ and assume now that the inductive assertion
holds true for all partitions of size smaller than $n$ and all
integers $i \geq 1$. Let us fix two partitions $\mu,\nu \vdash n$ and
an integer $i \geq 1$. Let $M$ be a map with the root vertex of
degree $i$, and the black (white, respectively) vertex distribution $\mu$ ($\nu$, respectively). We are going to understand the structure of $M\setminus\{e\}$, where $e$ is the root edge of $M$. There are two possibilities:
\begin{itemize} 
\item
$M\setminus\{e\}$ is a disjoint sum of two maps $M_1$ and $M_2$ (they
are ordered, i.e.~their indices matter) with root vertices of degrees
$i-1$, and $j-1$, respectively, the black vertex distributions
$\mu^1$, and $\mu^2$, respectively, and the white vertex distributions $\nu^1$, and $\nu^2$, respectively, where 
\begin{align*} 
&(\mu^1,\mu^2) \in \Sp^{(l,n-l-1)}(\mu_{\downarrow (i)}), \\
&(\nu^1,\nu^2) \in \Sp^{(l,n-l-1)}(\nu_{\downarrow (j)}),
\end{align*}
and $1 \leq j,l+1 \leq n$ are some integers;
\item
$M\setminus\{e\}$ is a single map $M'$ with the root vertex of degree
$i-1$, the black vertex distribution $\mu_{\downarrow (i)}$, and the white vertex distribution $\nu_{\downarrow (j)}$, where $2 \leq j \leq n$ is some integer.
\end{itemize}
Moreover,
\begin{itemize} 
\item
for any ordered pair of maps $M_1$, and $M_2$ with root vertices of
degrees $i-1$, and $j-1$, respectively, the black vertex distribution
$\mu^1$, and $\mu^2$, respectively, and the white vertex distribution $\nu^1$, and $\nu^2$, respectively, where 
\begin{align*} 
&(\mu^1,\mu^2) \in \Sp^{(l,n-l-1)}(\mu_{\downarrow (i)}), \\
&(\nu^1,\nu^2) \in \Sp^{(l,n-l-1)}(\nu_{\downarrow (j)}),
\end{align*}
and $1 \leq j,l+1 \leq n$ are some integers, there exists the unique
map $M$ with the root vertex of degree $i$ and the black (white, respectively) vertex distribution $\mu$ ($\nu$, respectively), such that $M^1 = M_1 \cup M_2$. In this case $\eta(M) = \eta(M_1) + \eta(M_2)$;
\item
for any map $M'$ with the root vertex of degree $i-1$, the black
vertex distribution $\mu\setminus (i) \cup (i-1)$ and the white vertex distribution $\nu\setminus (j) \cup (j-1)$, where $2 \leq j \leq n$ is some integer, there exists 
\[ 2(j-1)m_{j-1}(\nu)+2(j-1)\] 
maps with the root vertex of degree $i$ and the black (white,
respectively) vertex distribution $\mu$ ($\nu$, respectively), such
that removing its root edge gives a map $M'$. Indeed, each map with
$n$ edges and these properties is obtained by adding an edge to $M'$,
which connects the root corner $r$ of $M'$ to some corner $c$ of $M'$
incident to a white vertex of degree $j-1$. There are $(j-1)m_{j-1}(\nu) + (j-1)$ such corners (since there are $m_{j-1}(\nu) + 1$ white vertices of degree $j-1$ in the map $M'$) and for each chosen corner there are exactly two ways to connect it with the root corner of $M'$ by an edge (these two ways correspond to construction of maps $M$ and $\tau_1 M$ -- we recall that $\tau_1 M$ is a map obtained from $M$ by twisting its root edge; see \cref{def:twist}). Now, notice that there are following possibilities:
\begin{itemize}
\item
if the root corner $r$ of $M'$ and the corner $c$ belong to the same
face of $M'$ then exactly one rooted bipartite map from $\{M,\tau_1
M\}$ has the twisted root edge, while the second one has the root edge which is a border. Thus
\[\{\eta(M), \eta(\tau_1 M)\} = \{\eta(M'), \eta(M')+1\};\]
\item
if the root corner $r$ of $M'$ and the corner $c$ belong to different faces of $M'$ then both root edges of $M$ and $\tau_1 M$ are handles. Thus, strictly from the definition of $\eta$, one has \[\{\eta(M), \eta(\tau_1 M)\} = \{\eta(M'), \eta(M')+1\}.\]
\end{itemize}
\end{itemize}

Above analysis leads us to the following recursion obtained by removing root edges from the maps appearing in the summation index in the definition of $\left(H_\eta\right)_{\mu,\nu;i}(\beta)$ given by \cref{eq:LaCroixMeasure}:
\begin{multline*}
\left(H_\eta\right)_{\mu,\nu;i}(\beta) = \sum_{1 \leq j,l \leq n}
\sum_{\substack{(\mu^1,\mu^2) \in \Sp^{(l-1,n-l)}(\mu_{\downarrow (i)}), \\
(\nu^1,\nu^2) \in \Sp^{(l-1,n-l)}(\nu_{\downarrow (j)})}}
\left(H_\eta\right)_{\mu^1,\nu^1;i-1}(\beta)\left(H_\eta\right)_{\nu^2,\mu^2;j-1}(\beta)\\
+ (1+\beta)\sum_{2 \leq j \leq n}(j-1)(m_{j-1}(\nu) + 1)\left(H_\eta\right)_{\mu_{\downarrow (i)}, \nu_{\downarrow (j)};i-1}(\beta).
\end{multline*}
Using the inductive assertion, we obtain:
\begin{multline*}
\left(H_\eta\right)_{\mu,\nu;i}(\beta) = (1+\beta)^{n+1-\ell(\mu)-\ell(\nu)} \left( \sum_{1 \leq j,l \leq n}
\sum_{\substack{(\mu^1,\mu^2) \in \Sp^{(l-1,n-l)}(\mu_{\downarrow (i)}), \\
(\nu^1,\nu^2) \in \Sp^{(l-1,n-l)}(\nu_{\downarrow (j)})}}\tilde{H}_{\mu^1,\nu^1;i-1}\tilde{H}_{\nu^2,\mu^2;j-1}\right.\\
\left. + \sum_{2 \leq j \leq n}(j-1)(m_{j-1}(\nu) + 1)\tilde{H}_{\mu_{\downarrow (i)}, \nu_{\downarrow (j)};i-1}\right).
\end{multline*}
To finish the proof, it is enough to notice that the following recursion holds true:
\begin{multline*}
\tilde{H}_{\mu,\nu;i} = \left( \sum_{1 \leq j,l \leq n}
\sum_{\substack{(\mu^1,\mu^2) \in \Sp^{(l-1,n-l)}(\mu_{\downarrow (i)}), \\
(\nu^1,\nu^2) \in \Sp^{(l-1,n-l)}(\nu_{\downarrow (j)})}} \tilde{H}_{\mu^1,\nu^1;i-1}\tilde{H}_{\nu^2,\mu^2;j-1}\right.\\
\left.
+ \sum_{2 \leq j \leq n}(j-1)(m_{j-1}(\nu) + 1)\tilde{H}_{\mu_{\downarrow (i)}, \nu_{\downarrow (j)};i-1}\right).
\end{multline*}
Above relation comes from the analysis of the process of removing the
root edge from an orientable map with the root vertex of degree $i$,
the black vertex distribution $\mu$, and the white vertex distribution $\nu$. Such analysis is almost identical to the analysis we did in the general case, and we leave it as an easy exercise.
\end{proof}

\subsection{Some consequences of the polynomiality and the marginal sum results}

We start with an observation that polynomials $h_{\mu, \nu}^\tau(\beta)$ have some specific form:

\begin{lemma}
\label{lem:expansion}
For any positive integer $n$ and any partitions $\mu,\nu, \tau \vdash n$ one has the following expansion:
\begin{align*} 
h_{\mu, \nu}^\tau(\beta) = \begin{cases}\sum_{0 \leq i \leq [g/2]} a_{\mu,\nu;i}^\tau \beta^{g-2i}(\beta+1)^i &\text{ for } \ell(\mu)+\ell(\nu)+\ell(\tau) \leq 2 + n, \\
0 &\text{ otherwise }; \end{cases}
\end{align*}
where $g := 2 + n - \left(\ell(\mu)+\ell(\nu)+\ell(\tau)\right)$ and $a_{\mu,\nu;i}^\tau \in \QQ$.
\end{lemma}

\begin{proof}
La Croix proved in \cite[Corollary 5.22]{LaCroix2009} that \cref{lem:expansion} holds true, assuming polynomiality of $h_{\mu, \nu}^\tau(\beta)$ (that he was not able to prove). \cref{theo:Polynomiality} completes the proof.
\end{proof}

We are now ready to prove \cref{theo:UnhandledOnefaceMaps}

\begin{proof}[Proof of \cref{theo:UnhandledOnefaceMaps}]
\cref{theo:MapSeries} already says that the cases $\beta = 0,1$
correspond to counting maps on orientable, and all (orientable or non-orientable),
respectively, surfaces. Thus, we need to establish the remaining identity:
\[ h_{\mu,\nu}^{(n)}(-1) = \left(H_\eta\right)_{\mu,\nu}^{(n)}(-1).\]
Note now that for fixed partitions $\mu,\nu \vdash n$ the variable $g := 2 + n - \left(\ell(\mu) + \ell(\nu) + \ell(\tau)\right)$ taken over all partitions $\tau \vdash n$ realizes a maximum for $\tau = (n)$. Hence
\begin{multline*} [\beta^{1 + n - \left(\ell(\mu) + \ell(\nu)\right)}]\sum_{\tau \vdash n}h_{\mu,\nu}^{\tau}(\beta) = [\beta^{1 + n - \left(\ell(\mu) + \ell(\nu)\right)}]h_{\mu,\nu}^{(n)}(\beta) = (-1)^{1 + n - \left(\ell(\mu) + \ell(\nu)\right)}h_{\mu,\nu}^{(n)}(-1)
\end{multline*}
by \cref{lem:expansion} and, similarly,
\begin{multline*} [\beta^{1 + n - \left(\ell(\mu) + \ell(\nu)\right)}]\sum_{\tau \vdash n}\left(H_\eta\right)_{\mu,\nu}^{\tau}(\beta) = [\beta^{1 + n - \left(\ell(\mu) + \ell(\nu)\right)}]\left(H_\eta\right)_{\mu,\nu}^{(n)}(\beta) \\
= (-1)^{1 + n - \left(\ell(\mu) + \ell(\nu)\right)}\left(H_\eta\right)_{\mu,\nu}^{(n)}(-1)
\end{multline*}
by \cref{eq:HPolynomialForm} and \cref{lem:LaCroixPolynomials}.
By \cref{ptheo:marginalSumEta} the following equality holds
\[ [\beta^{1 + n - \left(\ell(\mu) + \ell(\nu)\right)}]\sum_{\tau \vdash n}h_{\mu,\nu}^{\tau}(\beta) = [\beta^{1 + n - \left(\ell(\mu) + \ell(\nu)\right)}]\sum_{\tau \vdash n}\left(H_\eta\right)_{\mu,\nu}^{\tau}(\beta)\]
which implies the desired result.
\end{proof}

\begin{remark}
The equality
\[ [\beta^{1 + n - \left(\ell(\mu) + \ell(\nu)\right)}]\sum_{\tau \vdash n}h_{\mu,\nu}^{\tau}(\beta) = [\beta^{1 + n - \left(\ell(\mu) + \ell(\nu)\right)}]\sum_{\tau \vdash n}\left(H_\eta\right)_{\mu,\nu}^{\tau}(\beta),\]
combined with \cref{prop:marginalsum}, \cref{lem:LaCroixPolynomials}
and  \cref{lem:expansion} says that the top degree coefficient of
$h_{\mu,\nu}^{(n)}(\beta)$ is enumerated by \emph{unhandled} maps of
type $(\mu,\nu;(n))$, but it is also enumerated by \emph{orientable}
maps with the black (white, respectively) vertex distribution $\mu$
($\nu$, respectively) and the arbitrary face degree. In fact, one can use the proof of \cref{prop:marginalsum} to construct a bijection between these two sets recursively.
An interesting result of \'Sniady \cite[Corollary 0.5]{Sniady2015} states that the top-degree part of Jack characters indexed by a one-part partition can be also expressed as a linear combination of certain functions indexed by orientable maps. \'Sniady informed us in private communication \cite{SniadyPrivate} that he can construct a similar bijection, but for different ``measure of non-orientability'', which inspired us to initiate a research presented in this section. The connection between striking similarities in both results seems to be far from being understood.
\end{remark}

\section{Low genera cases and orientations}
\label{sec:lowGenera}

In this section we are going to prove \cref{theo:LowGenera}. In fact, we are going to show that there is an infinite family of measures of non-orientability for which, in low genera cases, $b$-conjecture holds true.

\subsection{A measure of non-orientability given by orientations} 

Let $M$ be a map. We say that $O$ is an \emph{orientation} of $M$ if it defines an orientation of each face of $M$ such that
\begin{itemize}
\item
the orientation of the root face given by $O$ is consistent with the orientation given by the root;
\item 
if $M$ is orientable then $O$ is the canonical orientation of $M$, that is the orientation for which each face of $M$ is oriented clockwise (counterclockwise, respectively) iff the root face is oriented clockwise (counterclockwise, respectively).
\end{itemize}

Let $\mathcal{O}$ be a \emph{set of orientations of all rooted maps}
(\emph{set of orientations}, for short), i.e.~ for any map $M$ there exists the \emph{unique} orientation $O$ of $M$ such that $O \in \mathcal{O}$. We are going to define a function $\eta_{\OOO}$ associated with $\OOO$ that takes as values maps and returns a nonnegative integer that, in some sense, ``measures non-orientability'' of the given map. This function will be defined recursively using the same procedure of deleting the root edge from the given map, as in \cref{def:MeasureOfNonOrient}.

\begin{definition}
\label{def:MeasureForOrientation}
Let $\mathcal{O}$ be a set of orientations.
We set $\eta_\OOO(M) = 0$ for $M$ without edges, we fix a positive
integer $n$, and we assume that
$\eta_\OOO(M)$ is already defined for all maps with at most $n-1$
edges. Let $M$ be a map with $n$ edges and
let $O \in \OOO$ be the orientation associated with $M^2$
(in the case where $M^2$ is a disjoint sum of two maps
$M_1, M_2$, we are taking two associated orientations $O_1, O_2 \in
\OOO$, respectively). Let $c'$ be the unique corner of $M^2$ containing the first corner $c$ of $M$ visited  after the root
corner of $M$ and $c'$ inherits an orientation from the corner
$c$. We set $\eta_\OOO(M) := \eta_\OOO(M^2)$ if the
orientation of $c'$ is consistent with the orientation given by $O$
and we say that \emph{$e$ is of the first kind}; otherwise we set
$\eta_\OOO(M) := \eta_\OOO(M^2)+1$ and we say that
\emph{$e$ is of the second kind} (in the case where $M^2
= M_1 \cup M_2$ is a disjoint sum of two rooted maps we set, by
convention, $\eta_\OOO(M) := \eta_\OOO(M_1) +
\eta_\OOO(M_2)$ and we set $e$ to be of the first kind).
\end{definition}

It is easy to see that for each set of orientations $\OOO$, the following holds true:
\begin{itemize}
\item
if the root edge $e$ of $M$ is a \emph{bridge} then $\eta_\OOO(M) = \eta_\OOO(M \setminus \{e\})$,
\item
if the root edge $e$ of $M$ is a \emph{border} then $\eta_\OOO(M) = \eta_\OOO(M \setminus \{e\})$,
\item
if the root edge $e$ of $M$ is a \emph{twisted} edge then $\eta_\OOO(M) = \eta_\OOO(M \setminus \{e\})+1$,
\item
if the root edge $e$ of $M$ is a \emph{handle} then $\{\eta_\OOO(M), \eta_\OOO(
\tau_1 M)\} = \{\eta_\OOO(M \setminus \{e\}), \eta_\OOO(M \setminus \{e\}) + 1\}$.
\end{itemize}

In other words, for any orientation $\OOO$ of all rooted maps, the
measure of non-orientability $\eta_\OOO$ associated with $\OOO$ is
also a measure of non-orientability given by
\cref{def:MeasureOfNonOrient}. However, the converse statement is not
true, i.e. there are many measures of non-orientability $\eta$ which
are not given by any set of orientations of all rooted maps.

\begin{remark}
Note that we use \textbf{two different} ways to make a distinction
between edges: the
first way is by determining their \emph{type} which can be a bridge, a
border, a twisted edge, or a handle. The second way is by determining
their \emph{kind} which can be the first or the second. Each
\emph{type} of edges has uniquely determined \emph{kind} except a handle which
can be both of the first and of the second kind. With this notation,
invariant $\eta_\OOO(M)$ associated with the rooted map $M$ is equal to the number of edges of the second
kind appearing in its root-deletion process.
\end{remark}

We are ready to restate \cref{theo:LowGenera} in the more general form:

\begin{theorem}
\label{theo:UnicellularLowGenera}
For a set $\OOO$ of orientations, for any positive integer $n$, and for any partitions $\mu,\nu \vdash n$ such that $\ell(\mu) + \ell(\nu) \geq n - 3$
the following equality holds true:
\[h_{\mu,\nu}^{(n)}(\beta) = \left(H_{\eta_\OOO}\right)_{\mu,\nu}^{(n)}(\beta).\]
\end{theorem}

The next subsections are devoted to its proof.

\subsection{Unicellular maps of low genera with two handles}

In this section we are going to analyze the structure of the
unicellular maps of low genera with two handles, which is necessary
for the proof of \cref{theo:UnicellularLowGenera}. Since this section is very technical, we would like to precede it by
a description of the main idea of the proof of
\cref{theo:UnicellularLowGenera}. We should try to keep it light and a
bit informal to motivate the reader to understand all the
technicalities that will appear after this introduction and that are
necessary to present the formal proof of
\cref{theo:UnicellularLowGenera}.

\subsubsection{General idea}
Let $\OOO$ be a set of orientations and let us fix a
positive integer $n$ and partitions $\mu, \nu \vdash n$ such that
$\ell(\mu) + \ell(\nu) = n-3$. The most important construction in this
section gives an involution
\[ \sigma_\OOO : \widetilde{\M}^{(n)}_{\mu,\nu;2} \to \widetilde{\M}^{(n)}_{\mu,\nu;2}\]
such that
\[ \eta_\OOO\left(\sigma_\OOO(M)\right) = 2 - \eta_\OOO(M). \]
It is an easy exercise (later on it will be explained in details) that
having the above mentioned involution we can use a simple polynomial
interpolation argument
to prove \cref{theo:UnicellularLowGenera}, thus in the following we
are going to focus on the construction of such involution. 

Firstly, we need to
understand how a map $M \in \widetilde{\M}^{(n)}_{\mu,\nu;2}$ can look
like. This map has a unique face, genus $2$, and exactly $2$ handles
appear during its root-deletion process that correspond to edges
$e_i(M)$ and $e_j(M)$ with labels $j > i$. Moreover,
\cref{lem:NrOfHandles} asserts that no twisted edges appear during the
root-deletion process of $M$, thus $\eta_\OOO(M) \in \{0,1,2\}$ and it
depends only on the fact wether the root edges of $M^i$ and $M^j$ are
handles of the first kind or of the second kind. The most natural
idea of how to construct the
map $\sigma_\OOO(M)$ is to reverse somehow the root-deletion process
of $M$ in a way that the edges of $M$ and $\sigma_\OOO(M)$ with
the same labels have the same types and such that the root edges of
$\sigma_\OOO(M)^i$ and $M^i$ are
    handles of different kinds. Similarly the root edges of
$\sigma_\OOO(M)^j$ and $M^j$ are
    handles of different kinds. This suggests that the map
    $\sigma_\OOO(M)$ might be constructed by twisting some of the
    edges of $M$ in the appropriate way, similar to how it was done in
    \cref{lem:involution}. However, it turned out that in some cases
    it is impossible and one needs a deeper analysis of
    the structure of the map $M$. In the following section, we present
    an example where twisting some edges never works, and we
    show on this specific example how to
    overcome this problem. We believe that this example is the accurate
    toy-example of the
    general case.

\subsubsection{Example}
\label{subsec:example}

\begin{figure}
\centering
\subfloat[]{
	\label{subfig:m1}
	\includegraphics[width=0.32\linewidth]{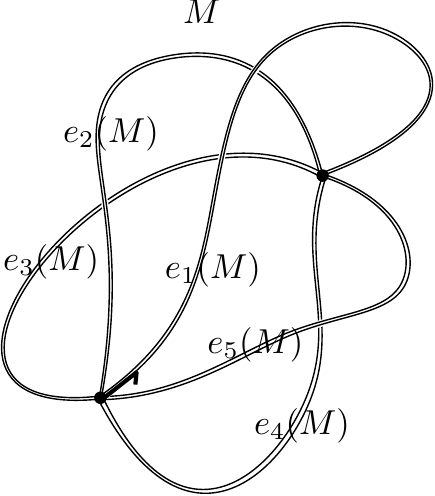}}
\subfloat[]{
	\label{subfig:m2}
	\includegraphics[width=0.32\linewidth]{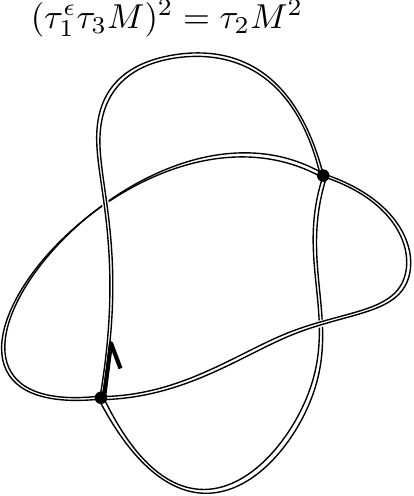}}
\subfloat[]{
	\label{subfig:m3}
	\includegraphics[width=0.32\linewidth]{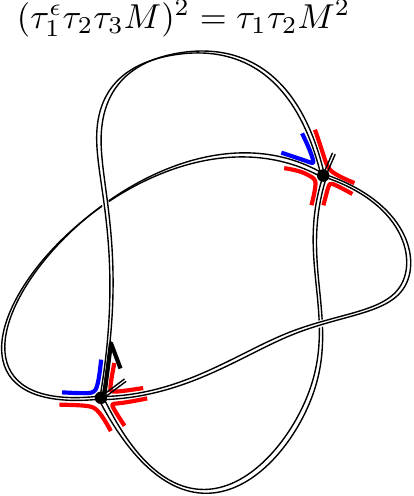}}
\caption{\protect\subref{subfig:m1} shows the map $M \in
\widetilde{\M}^{(5)}_{(5),(5);2}$ analyzed in this
section. \protect\subref{subfig:m2} shows that the map
$(\tau_1^\epsilon)\tau_3M)^2$ is unicellular and its root edge is
twisted. \protect\subref{subfig:m3} shows the map
$(\tau_1^\epsilon)\tau_2\tau_3M)^2$, which has two faces and the
corners lying in the same face have the same color.}
\label{fig:Przyklad'}
\end{figure}

Let us consider the following map $M \in
\widetilde{\M}^{(5)}_{(5),(5);2}$ presented in
\cref{subfig:m1}. Types of the root edges of the consecutive
maps obtained in the root-deletion process of $M$ are as follows:
$e_1(M)$ is a handle (of the first kind), $e_2(M)$ is a border, $e_3(M)$ is a
handle (of the first kind), $e_4(M)$ is a border again, and $e_5(M)$ is a
bridge. We would like to construct $\sigma_\OOO(M)$ by reversing the
root-deletion process of $M$ somehow in a way that:
\begin{itemize}
\item
the edges of $M$ and $\sigma_\OOO(M)$ with
the same labels have same types
\item the root edges of
$\sigma_\OOO(M)^3$ and $M^3$ are
    handles of different kinds. Similarly, the root edges of
$\sigma_\OOO(M)^1$ and $M^1$ are
    handles of different kinds.
\end{itemize} 
Thus, $\sigma_\OOO(M)^5 = M^5$ and $\sigma(M)^4=M^4$. Since the root edges of $\sigma_\OOO(M)^3$ and $M^3$ are handles
    of different kinds, $\si_\OOO(M)^3$ is obtained from $M^3$ by twisting its root edge. Thus, if we would like to build $\sigma(M)$ by
    twisting edges of $M$, it has to be of the following form: $\sigma(M)
    = \tau_1^{\epsilon_1}\tau_2^{\epsilon_2}\tau_3 M$, where
    $\epsilon_1, \epsilon_2 \in \{0,1\}$. However, for all these
    possible choices of $\epsilon_1,\epsilon_2$ the map
    $\tau_1^{\epsilon_1}\tau_2^{\epsilon_2}\tau_3 M$ does not satisfy
    above required properties. Indeed, the root
    edge of the map $(\tau_1^{\epsilon_1}\tau_2M)^2 = \tau_2 M^2$ is
    twisted. See \cref{subfig:m2}, which is a problem since we want this
    root edge to be a border. If we twist it, that is we
    consider the map $(\tau_1^{\epsilon_1}\tau_2\tau_3M)^2 = \tau_1\tau_2 M^2$, then its root edge is a
    border, so we fix previous problem. However, we encounter
    another case: half-edges $h_1(\tau_1^{\epsilon_1}\tau_2\tau_3M)$
    and $h_1 (\tau_1^{\epsilon_1}\tau_2\tau_3M)'$ lie in the same face of $\tau_1\tau_2 M^2$, thus
    the root edge of $\tau_1^{\epsilon_1}\tau_2\tau_3M$ cannot be a
    handles, see
    \cref{subfig:m2}. 

A crucial observation which helps to overcome the problem is the following: the set of corners
lying in the root face of $\tau_1\tau_2 M^2$ differs from the set of
corners lying in the root face of $M^2$ (compare \cref{subfig:m3} to
\cref{subfig:prz1}). We would like to fix this. One can show that if we erase the edges $e_1(M^2)$ and $e_2(M^2)$ from the map $M^2$, but
we do not erase the corresponding half-edges $h_1(M^2)$, $h_1(M^2)'$, $h_2(M^2)$, $h_2(M^2)'$, then there is a unique
way to choose two pairs from these four half-edges and draw two new 
edges connecting half-edges in each pair to obtain a map $M'$ different
from $M^2$ such that:
\begin{itemize}
\item
its
root edge is a border,
\item
the set of corners lying in its root face
is exactly the same as the set of corners lying in the root face of
$M^2$.
\end{itemize}
This map is shown in \cref{subfig:prz2}. Moreover, the root
edge of $(M')^2$ is a handle of the different kind than the root edge of
$M^3$. Now, we can finish the construction of $\sigma_\OOO(M)$. The set of corners
lying in the root face of $M'$ is the same as the set of corners lying in the root face of
$M^2$. Thus, if we connect the corners of $M'$ corresponding to
the root edge of $M$ by a new edge we will always create a handle (of two
possible kinds). Thus, it
is enough to connect these corners by a handle of different kind than
the root edge of $M$ to construct $\sigma_\OOO(M)$. In our case it is
a handle of the second kind and this map is shown in
\cref{subfig:prz3}. It is straightforward from this construction that
if we repeat this recipe to construct $\si_\OOO(\si_\OOO(M))$, we construct exactly the map $M$.

\begin{figure}
\centering
\subfloat[]{
	\label{subfig:prz1}
	\includegraphics[width=0.32\linewidth]{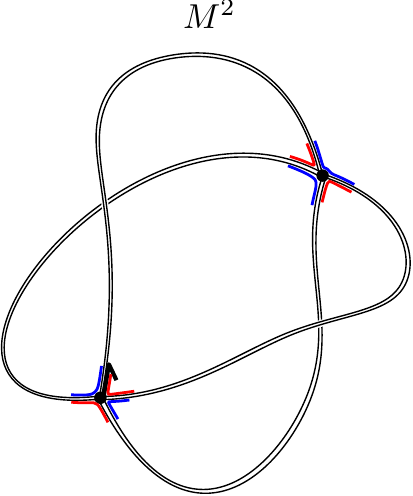}}
\subfloat[]{
	\label{subfig:prz2}
	\includegraphics[width=0.32\linewidth]{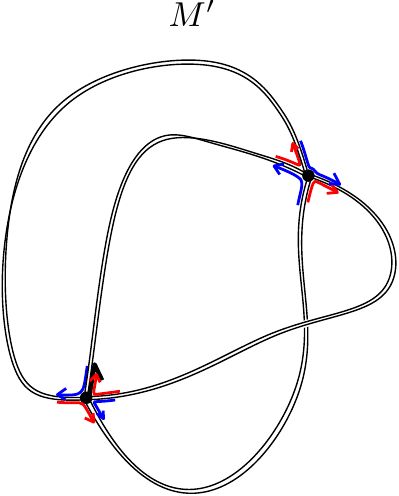}}
\subfloat[]{
	\label{subfig:prz3}
	\includegraphics[width=0.32\linewidth]{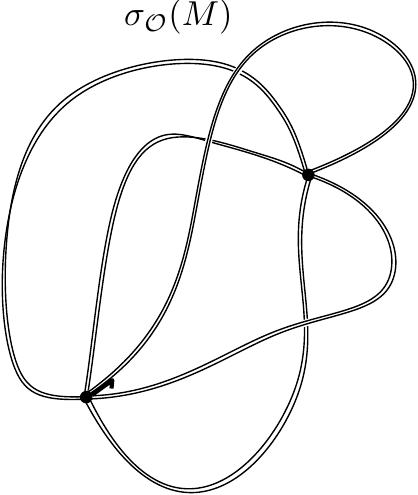}}
\caption{Comparing \protect\subref{subfig:prz1} to
  \protect\subref{subfig:prz2} we can see that both maps $M^2$ and
  $M'$ have two faces, and the corners in the same face have the same
  colors (red or blue), thus $C(M^2) = C(M')$. \protect\subref{subfig:prz3} shows the map
$\si_\OOO(M)$ obtained from $M'$ by connecting $h_1(M)$ with $h_1(M')$
in the appropriate way.}
\label{fig:Przyklad}
\end{figure}

It turned out that the general situation is basically the same. The next section is
devoted to the description of the structure of a map $M$ in a general
case when we cannot construct $\si_\OOO(M)$ simply by twisting some
edges of $M$. This description, which is given in
\cref{lem:order} and its proof are very technical. However, the reader should think that a general picture looks 
almost the same as the picture from this section:
a map $M$ might have a lot of edges, but only three edges play
important role in the construction of $\sigma_\OOO(M)$ and their role
is the same as the role of $e_1(M),e_2(M)$, and $e_3(M)$ in the above example.

\subsubsection{Details}
\label{subsec:details}

Let us fix a positive integer $n$, and partitions $\mu,\nu \vdash n$
such that $\ell(\mu) + \ell(\nu) = n-3$. Let $M \in
\widetilde{\M}^{(n)}_{\mu,\nu;2}$. Then \cref{eq:pomocnicze} states that there are no twisted edges in the root-deletion process
of $M$ (since $\ell(\mu)+\ell(\nu) = n-3$). Thus, for all positive
integers $k <i$
the root edge of $M^k$ is a bridge and there are two possible situations:

\begin{itemize}
\item
for all positive integers $i < k < j$ the root edge of $M^k$ is a bridge;
\item
there exists a unique positive integer $i < k < j$ such that the root
edge of $M^k$ is a border. Then there are still two possible cases:
\begin{itemize}
\item
$C(M^k) = C((\tau_j M)^k)$;
\item
$C(M^k) \neq C((\tau_j M)^k)$.
\end{itemize}
\end{itemize}

Note that the map $M$ presented in
\cref{subsec:example} was exactly the second case of the second case
in the above analysis. We start with a technical lemma that treats
this case in
general. All the symbols $n,i,j, \mu,\nu$ are as above and we recall that for any $m
\in [n]$ the half-edge $h_m(M)$ is the root of $M^m$ and $h_m(M)'$ is the second half-edge belonging to the root edge $e_m(M)$.

\begin{lemma}
\label{lem:order}
Let $M \in \widetilde{\M}^{(n)}_{\mu,\nu;2}$. We assume that there
exists a positive integer $i < k < j$ such that the root edge of $M^k$
is a border and such that $C(M^m) = C((\tau_j M)^m)$ for all positive
integers $m > k$, but $C(M^k) \neq C((\tau_j M)^k)$. 
% If $k$ is the greatest possible integer with this property then:
Then:
\begin{enumerate}[label=(P\arabic*)]
% \item
% \label{item:lemma1}
% necessarily $i < k < j$, the root edge of $M^k$ is a border, and for all $m \in [j]$ different from $i,k,j$, the root edge of $M^m$ is a bridge;
\item
\label{item:lemma2}
if $M'$ is the map created from $M$ by erasing edges $e_k(M)$ and $e_j(M)$,
and merging pairs of half edges $\{h_k(M),h_j(M)'\}$ and
$\{h_j(M),h_k(M)'\}$ into edges $e_k'$ and $e_j'$, respectively (in
arbitrary way), then for all $m \in [n]$ the root edge of $M^m$ is a
bridge iff the root edge of $(M')^m$ is a bridge and $h_m(M) = h_m(M')$;
\item
\label{item:lemma3}
there is a unique way to construct a map $M'$ as in \ref{item:lemma2}
such that $C((M')^k) = C(M^k)$; then for all $m \in [n]$ types of
the root edges of $M^m$ and $(M')^m$ coincide and for any set $\OOO$
of orientations of all rooted maps the edges $e_j(M)$ and $e_j(M')$
are handles of different kinds. That is,
\[ \{\eta_\OOO(M^j), \eta_\OOO((M')^j)\} = \{0,1\}.\]
\end{enumerate}
\end{lemma}

\begin{proof}

Let $M'$ be a map constructed from $M$ as in
\ref{item:lemma2}. Strictly from the construction of root-deletion
process -- see \cref{subsec:TypesOfEdges} -- it is clear that if $m
\in [n]$, and the root edges of both $M^m$ and $(M')^m$ are not
bridges, then the roots of $M^{m}$ and $(M')^{m}$ coincide. It is
also clear that if the root edges of both $M^m$, and $(M')^m$ are bridges, and if the
second visited corner after the root corner in both maps $M^{m+1}$ and
$(M')^{m+1}$ coincide, then the root of $M^{m+1}$ and $(M')^{m+1}$
coincide, too.
So in order to prove \ref{item:lemma2} it is enough to show that for
all $m \in [n]$ the root edge of $M^m$ is a bridge iff the root edge
of $(M')^m$ is a bridge. We claim that 
\begin{enumerate}[label=($\star$)]
\item
\label{item:lemma4}
if one removes the edges containing $h_k(M)$ and $h_j(M)$ from $M^k$
then the resulting object $F$ is connected. Moreover, $F$ is obtained
by planting some maps into some corners of $M^{j+1}$.
\end{enumerate}
This claim easily implies the fact that for all $m \in [n]$ the root edge of $M^m$ is a bridge iff the root edge
of $(M')^m$ is a bridge. This may be shown by induction on
$m$. Assume that the root edge of $M^m$ is a bridge iff the root edge
of $(M')^m$ is a bridge for all $m < l \leq n$. This implies that for
all $m < l$ the root $h_m(M')$ of $(M')^m$ is equal to the root
$h_m(M)$ of $M^m$. Note that \ref{item:lemma4} implies that the number of
connected components of the graph $M\setminus\{e_1(M),\dots,e_l(M)\}$
is the same as the number of connected components of
$M\setminus\{e_k(M),e_j(M)\}\setminus\{e_1(M),\dots,e_l(M)\}$. But
strictly from the definition of $M'$, the last set is equal to
  $M'\setminus\{e_k',e_j'\}\setminus\{e_1(M'),\dots,e_l(M')\}$ which
 has the same number of connected components as
  $M'\setminus\{e_1(M'),\dots,e_l(M')\}$ by
  \ref{item:lemma4}. Thus, the root edge of $M^l$ is a bridge iff the root edge
of $(M')^l$ is a bridge, which finishes the proof of
\ref{item:lemma2}.

\begin{figure}
\centering
\subfloat[]{
	\label{subfig:Mk+1}
	\includegraphics[width=0.49\linewidth]{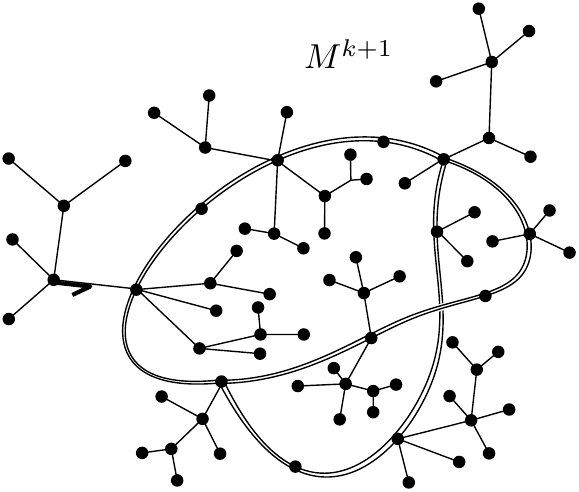}}
\subfloat[]{
	\label{subfig:Mk+1'}
	\includegraphics[width=0.49\linewidth]{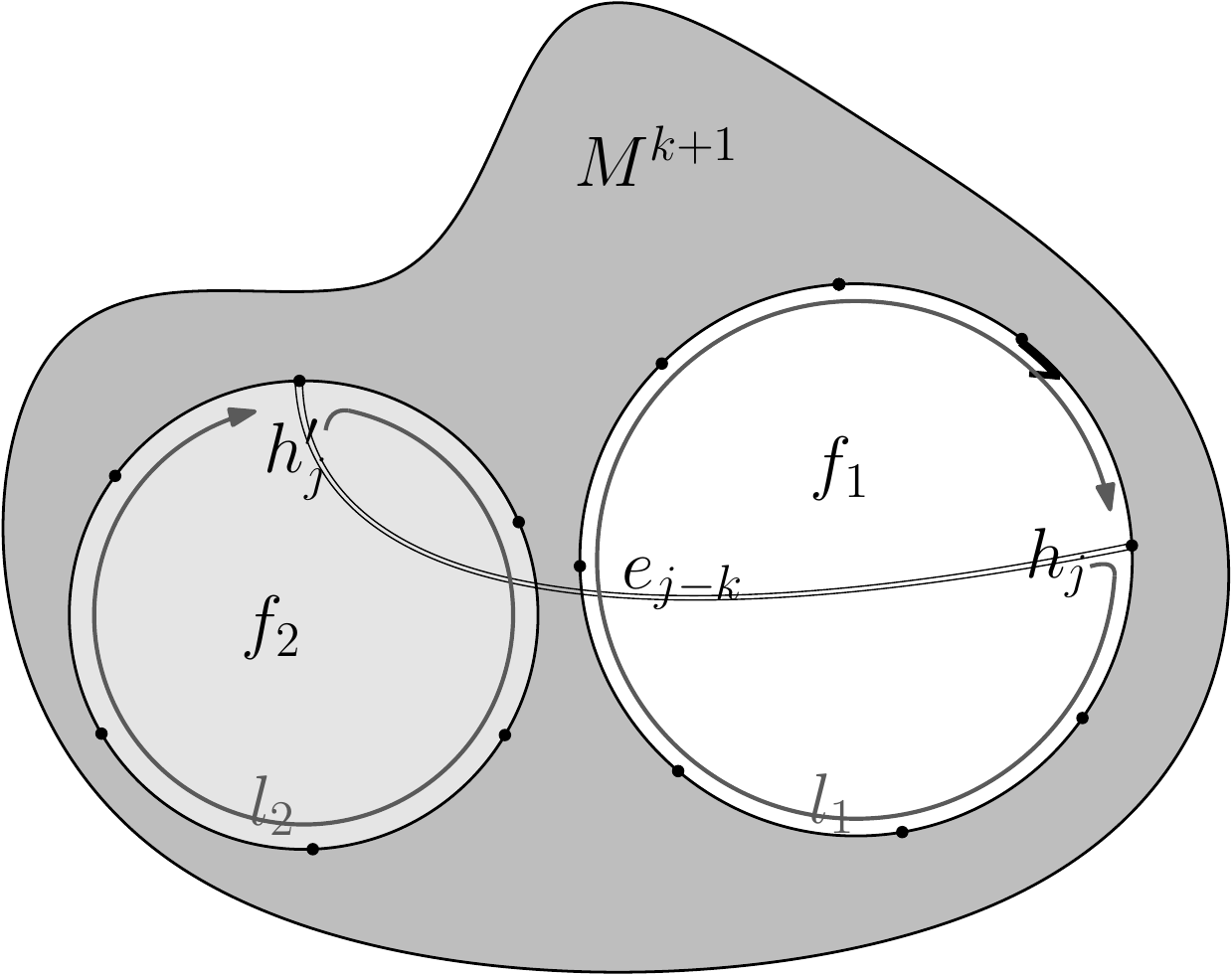}}
\caption{\protect\subref{subfig:Mk+1} shows that the unicellular map
  $M^{k+1}$ is obtained from the unicellular map $M^j$
  of genus $1$ by planting a collection of trees. \protect\subref{subfig:Mk+1'} represents
  diagrammatically the unicellular map $M^{k+1}$ -- the edge $e_j(M)$
  is its handle, and two areas $f_1,f_2$ represents two faces
  of the map $M^{k+1}\setminus\{e_j\left(M\right)\}$, and two
  oriented arcs $l_1,l_2$ correspond to the words given by
  reading consecutive corners visited in the root face.}
\label{fig:2handles}
\end{figure}

\medskip

In order to prove \ref{item:lemma4}, we need to analyze the structure
of the map $M^{k+1}$. This analysis will be also crucial in proving
\ref{item:lemma3}, (the structure of the map $M^{k+1}$ is shown on
\cref{fig:2handles}).
 First of all, it is clear from the classification of types of edges -- see \cref{subsec:TypesOfEdges} -- that 
\begin{itemize}
\item
for all $m \in [i]$ both maps $M^m$, and $(\tau_j M)^m$ are
unicellular,
\item
for $i < m \leq k$ both maps $M^m$, and $(\tau_j M)^m$ have two faces iff they contain a half-edge $h_k(M)$ (otherwise they are unicellular),
\item
for $k < m \leq j$ both maps $M^m$, and $(\tau_j M)^m$ are unicellular. 
\end{itemize}
As a result, all of the connected components of
$M\setminus\{e_1(M),\dots,e_{j-1}(M)\}$ are planar and unicellular
(thus they are trees),
except for the component containing the edge $e_j(M)$ which is also
unicellular but has genus equal to $1$. Indeed, it follows immediately
from Euler formula \cref{eq:euler} and from the definition of types of
edges given in \cref{subsec:TypesOfEdges}. Since the root edges of all
the maps $M^m$ are bridges for $k < m < j$, the map
$M^{k+1}$ is obtained from the unicellular map $M^j$ by planting some trees into some corners of it. In particular, the edge containing $h_j(M)$ has
the same type in both maps $M^j$ and $M^k$. Thus, it is a
handle. We conclude that $M^{k+1}\setminus\{e_j(M)\}$ is a map with two
faces -- in particular it is connected which proves our claim
\ref{item:lemma4}. 

\medskip

Above analysis says that if one removes the edge containing $h_j(M)$ from the map
$M^{k+1}$ but does not remove the corresponding half-edges $h_j(M)$
and $h_j(M)'$, the resulting object is a map with two faces $f_1$, and
$f_2$ and with two additional half-edges $h_j(M), h_j(M)'$ such that $h_j(M)$ lies in some corner belonging to $f_1$, and
$h_j(M)'$ lies in some corner belonging to $f_2$ (see
\cref{fig:2handles}). Let $l_1$ ($l_2$, respectively) be the word given
by reading consecutive corners of the face $f_1$ with additional
half-edge $h_j(M)$ ($f_2$ with additional half-edge $h_j(M)'$,
respectively) in a way that the word given by reading consecutive
corners of the unique face of the map $M^{k+1}$ with respect to the
root orientation is given by the concatenation $l_1\cdot l_2$ of the words
$l_1$ and $l_2$. See
\cref{fig:2handles}. For a given word $w$, we denote by
$\overleftarrow{w}$ the word obtained from $w$ by reading it backwards
(from right to left). Then the word given by reading consecutive
corners of the unique face of $(\tau_j M)^{k+1} = \tau_{j-k}\ M^{k+1}$
with respect to the root orientation is given by the concatenation
$l_1\cdot\overleftarrow{l_2}$.

\begin{figure}
\centering
\subfloat[]{
	\label{subfig:hard}
	\includegraphics[width=0.49\linewidth]{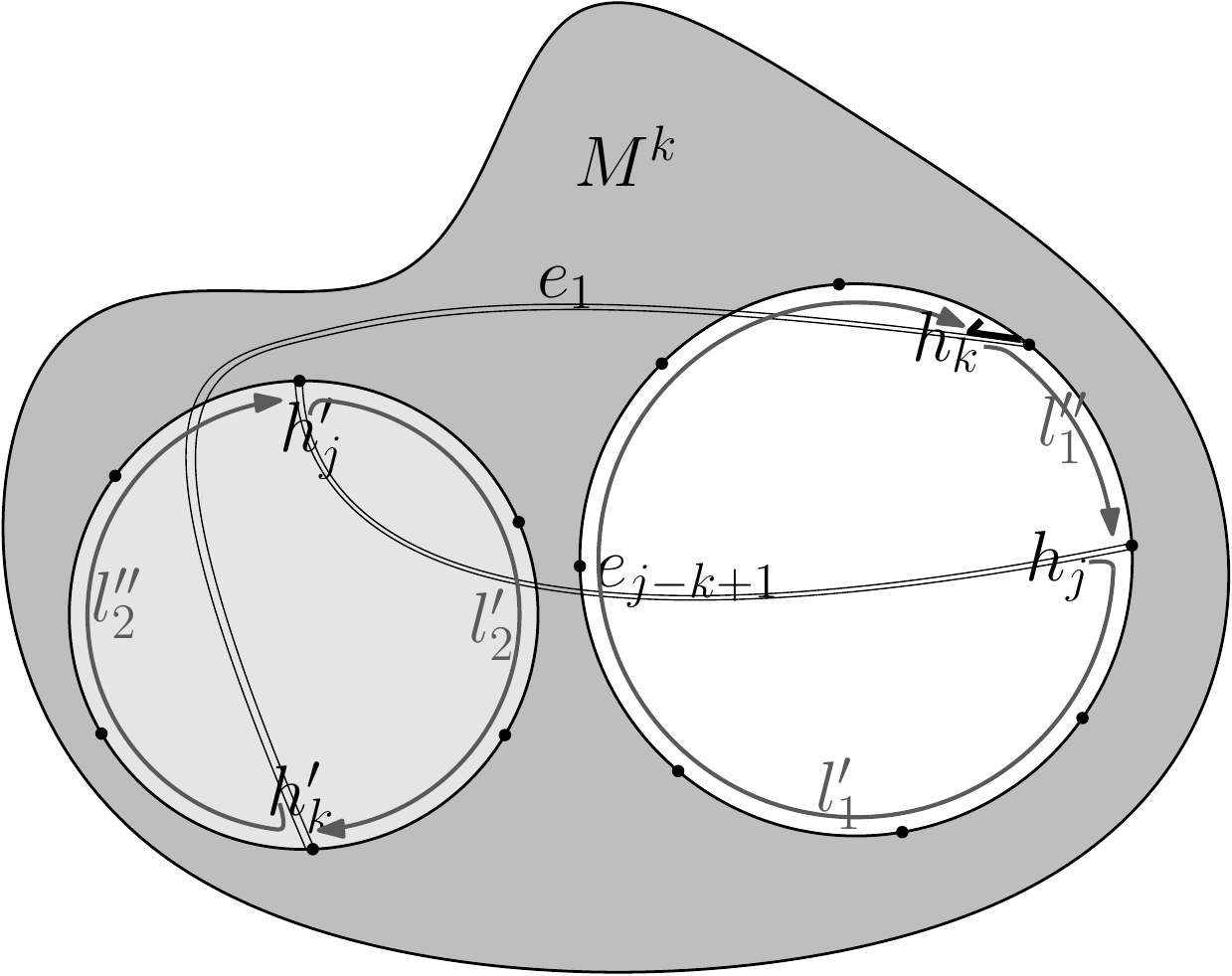}}
\subfloat[]{
	\label{subfig:hard'}
	\includegraphics[width=0.49\linewidth]{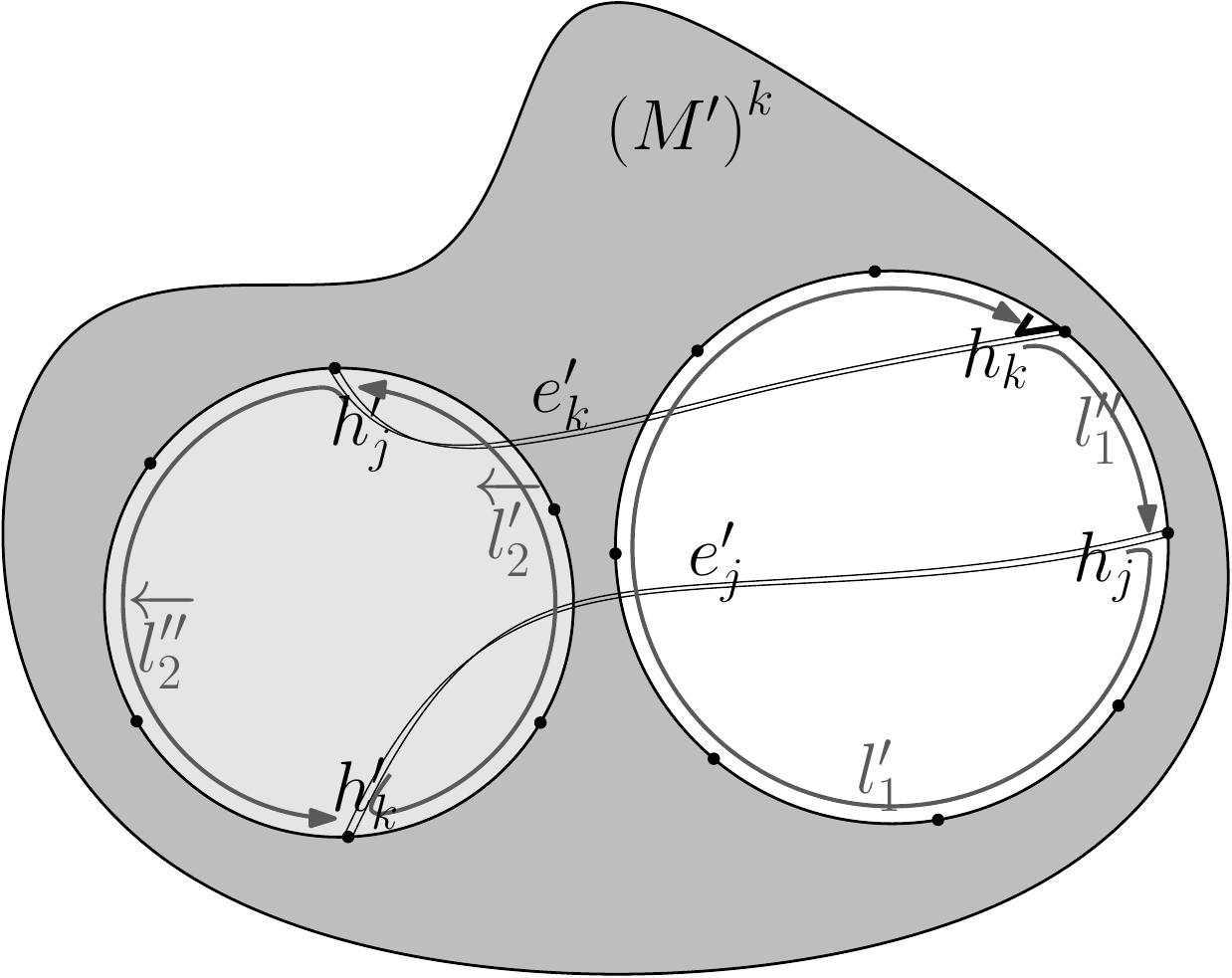}}
\caption{\protect\subref{subfig:hard} represents diagrammatically a
  map $M^k$ and \protect\subref{subfig:hard'} represents
  diagrammatically a map $(M')^k$. Indices of edges $e_k' =
  e_1\left((M')^k\right)$ and $e_j' = e_{j-k+1}\left((M')^k\right)$ distinguished in \protect\subref{subfig:hard'} correspond to their labels in $M'$. }
\label{fig:hard}
\end{figure}

Now, we recall our assumption that $C(M^k)\neq C((\tau_jM)^k)$. This
is equivalent to saying that the set of corners belonging to the root
face of $M^k$ is different than the set of corners belonging to the
root face of $(\tau_jM)^k$ (since both $M^k$ and $(\tau_jM)^k$ have
exactly two faces). Thus, the half-edge $h_k(M)$ is lying in
some corner belonging to $l_1$ and incident to $h_{k+1}(M)$ and $h_k(M)$ divides this corner into two new
corners. Similarly $h_k(M)'$ is lying in some corner belonging to $l_2$
and divides it into two new corners (indeed, if both half-edges
$h_k(M)$, and $h_k(M)'$ are lying in $l_1$, then, clearly, the sets
of corners belonging to the root face of $M^k$ and to the root face of
$(\tau_j M)^k$ coincide, which gives a contradiction with our
assumption). Let $l_1',l_1''$ and $l_2',l_2''$, respectively be two
new words obtained by reading consecutive corners between $h_k(M)$
and $h_j(M)$ and between $h_k(M)'$ and $h_j(M)'$, respectively, as
depicted in \cref{subfig:hard}. In other words, $l_1$ is a concatenation of
$(l_1')_-$, a letter $c_1$ which corresponds to the corner of
$M^{k+1}$ containing $h_k(M)$, and $_-(l_1'')$, where
$w_-$ ($_-w$, respectively) is a word obtained from $w$ by removing
its last (first, respectively) letter. Similarly, $l_2$ is a
concatenation of $(l_2')_-$, a letter $c_2$ which correspond to the
corner of $M^{k+1}$ containing $h_k(M)'$, and $_-(l_2'')$. Then the word given by reading consecutive
corners of the root face of $M^k$, starting from the first corner
visited after the root corner, is given by the concatenation
$l_2''\cdot l_1'$, while the word obtained by reading the
corners visited consecutively in the root face of $(\tau_j M)^k$,
starting from the first corner visited after the root corner, is given by
the concatenation $l_2''\cdot
\overleftarrow{l_1''}\cdot\overleftarrow{l_2'}\cdot l_1'$. In particular, the map $(\tau_j M)^k$ is unicellular, and its root is twisted, which shows that $C(M^k)\neq C((\tau_jM)^k)$, as we assumed. 

\medskip

Finally, in order to construct a map $M'$ as in \ref{item:lemma2} such
that $C((M')^k) = C(M^k)$ we need to merge $h_k(M)$ with $h_j(M)'$
and $h_j(M)$ with $h_k(M)'$ in a way such that the word given by reading
consecutive corners of the root face of $(M')^k$ with respect to the
root orientation and starting from the first visited corner after the
root corner is given by the concatenation of $\overleftarrow{l_2''}$
and $l_1'$. See \cref{subfig:hard'}. We recall that for $k+1 \leq m
\leq j$, if the unicellular map $M^m$ ($(M')^m$, respectively)
contains $h_j(M)$, then the edge containing $h_j(M)$ is a
handle. Moreover, if we equip two faces $f_1,f_2$ of
$M^m\setminus\{e_j\left(M^m\right)\}$ ($f_1',f_2'$ of
$(M')^m\setminus\{e_j\left(M^m\right)\}$, respectively), where $f_1$
($f_1'$, respectively) is the root face, into the orientation
inherited from the the root face of $M^m$ ($(M')^m$, respectively),
then the orientations of $f_1$, and $f_1'$ are the same, while the
orientations of $f_2$ and $f_2'$ are opposite to each other (compare
\cref{subfig:hard} to \cref{subfig:hard'}). By \cref{def:MeasureForOrientation}
\[ \{\eta_{\OOO}(M^j), \eta_{\OOO}((M')^j)\} = \{0,1\}. \]
Thus, $e_j(M)$ and $e_j(M')$ are handles of dfferent kinds.
In particular, for all integers $m \geq k$ the types of the root edges
of $M^m$ and $(M')^m$ are the same. Since $C((M')^k) = C(M^k)$ and since for all $m \in [k-1]$ the root edges of $M^m$ and $(M')^m$ contain the same pairs of half-edges and are not twisted nor bridges, they have the same types, which finishes the proof. 
\end{proof}

Now we are going to show that the assumption of
\cref{lem:order} that the root edge of $M^k$
is a border is in fact implied by the assumption that $C(M^k) \neq C((\tau_j M)^k)$.

\begin{lemma}
\label{lem:pomoc}
Let $M \in \widetilde{\M}^{(n)}_{\mu,\nu;2}$, let $i <j$ be the labels of
the edges that are handles in the root-deletion process of $M$, and let $k < j$ be the
largest positive integer such that $C(M^k) \neq C((\tau_j M)^k)$ (we
assume that it exists). 
% If $k$ is the greatest possible integer with this property then:
Then necessarily $i < k < j$ and the root edge of $M^k$ is a border.
\end{lemma}

\begin{proof}

We recall from the beginning of \cref{subsec:details} that there are
two possible cases:
\begin{itemize}
\item
for all positive integers $m \in [j]\setminus\{i, j\}$ the root edge of $M^m$ is a bridge;
\item
there exists the unique positive integer $i < k < j$ such that the root
edge of $M^k$ is a border and for all positive integers $m \in [j]\setminus\{i, k,j\}$ the root edge of $M^m$ is a bridge.
\end{itemize}

We are going to compare the sets $C(M^{m+1})$ and $C(M^m)$ in the following cases:
\begin{itemize}
\item
the root edge of $M^m$ is a bridge: there exists $l >
m+1$ such that $M^m\setminus\{e_m\left(M\right)\} = M^{m+1}\cup M^l$
and there exist two corners $c_1 \in f_1 \in C(M^{m+1})$, and $c_2 \in
f_2 \in C(M^l)$ which correspond to the root corners of $M^{m+1}$ and
$M^l$, respectively, and which are divided by $h_m(M)$ and $h_m(M)'$,
respectively, into two pairs of new corners $c_1',c_1''$, and $c_2',c_2''$ such that
\begin{align*} 
C(M^{m}) &= \underbrace{C(M^{m+1})\setminus\{f_1\}}_{\text{the set of corners belonging to the faces of $M^{m+1}$ different than the root face}} \\
&\cup \underbrace{C(M^l)\setminus\{f_2\}}_{\text{the set of corners belonging to the faces of $M^l$ different than the root face}}\\
&\cup \underbrace{\{f_1\setminus\{c_1\} \cup f_2\setminus\{c_2\}\cup
  \{c_1',c_1'',c_2'c_2''\}\}.}_{\text{the set of corners belonging to
  the root face of $M^m$ obtained by merging root faces of $M^{m+1}$ and $M^l$}}
\end{align*}
\item
the root edge of $M^m$ is a handle: there exist two
corners $c_1 \in f_1 \in C(M^{m+1})$, and $c_2 \in f_2 \in C(M^{m+1})$
containing $h_m(M)$ and $h_m(M)'$, respectively, such that $f_1 \neq
f_2$, and $h_m(M)$, $h_m(M)'$, respectively, divides $c_1$ and $c_2$
into two pairs of new corners $c_1',c_1''$, and $c_2',c_2''$,
respectively. Thus,
\begin{align*} 
C(M^{m}) &= \underbrace{C(M^{m+1})\setminus\{f_1,f_2\}}_{\text{the set
           of corners belonging to the faces of $M^{m+1}$ different
           from the faces merged by a handle $e_m(M)$}} \\
&\cup \underbrace{\{f_1\setminus\{c_1\} \cup f_2\setminus\{c_2\}\cup \{c_1',c_1'',c_2'c_2''\}\}.}_{\text{the set of corners belonging to the root face of $M^m$ obtained by merging faces $f_1$ and $f_2$}}
\end{align*}
\end{itemize}
If for all $m \in [j]\setminus\{i,j\}$ the root edge of $M^m$ is a
bridge then by \cref{prop:blabla} also the root edge of $(\tau_j M)^m$
is a bridge thus for all $m \leq j$ types of the root edges of
$M^m$ and $(\tau_j M)^m$ coincide and they are either bridges or
handles. Since for all $m > j$ the maps $M^m$ and $(\tau_j M)^m$ are
the same, above analysis gives immediately that $C(M^m) = C((\tau_j
M)^m)$ holds true for all $m \in [n]$. This proves that if there exists
$k \in [n]$ such that  $C(M^k) \neq C((\tau_j M)^k)$ and such that $k$
is the greatest possible, then necessarily $i < k < j$ and the root edge of $M^k$ is a border, which proves \cref{lem:pomoc}.

\end{proof}

We are finally ready to construct the promised involution.

\begin{lemma}
\label{lem:2handles}
Let $\OOO$ be a set of orientations and let us fix a positive integer $n$ and partitions $\mu, \nu \vdash n$ such that $\ell(\mu) + \ell(\nu) = n-3$. Then, there exists an involution
\[ \sigma_\OOO : \widetilde{\M}^{(n)}_{\mu,\nu;2} \to \widetilde{\M}^{(n)}_{\mu,\nu;2}\]
such that
\[ \eta_\OOO\left(\sigma_\OOO(M)\right) = 2 - \eta_\OOO(M). \]
\end{lemma}

\begin{proof}
Let us fix $M \in \widetilde{\M}^{(n)}_{\mu,\nu;2}$ and let $i <j$ be the labels of
the edges that are handles in the root-deletion process of $M$.

There are two possible cases: for all positive integers $m$ sets
$C(M^m)$ and $C((\tau_jM)^m)$ are the same, or not. If they coincide,
we consider two maps: $M_1 := \tau_i \tau_j M$, and $M_2 := \tau_j
M$. \cref{prop:blabla} ensures that labels in $M_1$, and $M_2$ are the
same as in $M$. Thus, types of the root edges of all three maps
$M^m, M_1^m$, and $M_2^m$ coincide, too. In particular, the root edges of $M^j$, and $M_1^j = M_2^j = \tau_1 (M^j)$ are
handles so we have 
\begin{equation}
\label{eq:mapy12}
\{\eta_{\OOO}(M^j), \eta_{\OOO}(\tau_1 M^j)\} = \{0,1\}.
\end{equation}
Moreover, the root edge of both $M_1^i$, and $M_2^i$ is a handle and
\begin{multline*}
\{\eta_\OOO(M_1^i), \eta_\OOO(M_2^i) \} = \{\eta_\OOO(\tau_1(\tau_j M)^i), \eta_\OOO((\tau_j M)^i) \}
= \{\eta_\OOO((\tau_j M)^{i+1}), \eta_\OOO((\tau_j M)^{i+1})+1 \}.
\end{multline*}
Combining this with \cref{eq:mapy12}, we obtain that there exists an integer $l \in \{1,2\}$ such that
 \[\eta_\OOO(M_l) = 2 - \eta_\OOO(M). \]
In other words, there exists $\epsilon \in \{0,1\}$ such that $M_l = \tau_i^\epsilon \tau_j M$. We define:
\[ \si_\OOO(M) := \tau_i^\epsilon \tau_j M.\]
Now, it is straightforward from the construction that $\si_\OOO(M) \in \widetilde{\M}^{(n)}_{\mu,\nu;2}$ and $\si_\OOO(\si_\OOO(M))$ is again of the same form. That is, $\si_\OOO(\si_\OOO(M)) := \tau_i^\epsilon \tau_j\si_\OOO(M)$. Thus, by \cref{prop:blabla},
\[ \si_\OOO(\si_\OOO(M)) = \tau_i^{2 \epsilon} \tau_j^2 M = M,\]
which finishes the proof in the first case.

\medskip

Now, we are going to treat the second case. That is, we assume that
there exists $k \in [n]$ such that $C(M^k) \neq C((\tau_j M)^k)$, and
we choose the greatest possible $k$ with this property. Let $M'$ be
the unique map as in \cref{lem:order} \ref{item:lemma3} associated with
$M$ and we define two maps $M_1 := M', M_2 := \tau_i M'$. First of
all, \cref{prop:blabla}, \cref{lem:order}, and \cref{lem:pomoc} state
that for all positive integers $m$ types of the root edges of all
three maps $M,M_1$ and $M_2$ coincide. Thus, $M,M_1,M_2 \in
\widetilde{\M}^{(n)}_{\mu,\nu;2}$. Moreover, \cref{prop:blabla} and \cref{lem:order} say that for all positive integers $m$ the roots of all three maps $M^m, M_1^m$, and $M_2^m$ are the same. Thus, for all positive integers $m > i$ maps $M_1^m$ and $M_2^m$ coincide and
\begin{equation}
\label{eq:mapy22}
\{\eta_{\OOO}(M^j), \eta_{\OOO}(M_1^j) = \eta_{\OOO}(M_2^j) = \eta_{\OOO}((M')^j)\} = \{0,1\},   
\end{equation}
by \cref{lem:order} \ref{item:lemma3}. Also
\begin{multline*}
\{\eta_\OOO(M_1^i), \eta_\OOO(M_2^i) \} = \{\eta_\OOO(\tau_1(M')^i), \eta_\OOO((M')^i) \} \\
= \{\eta_\OOO((M')^{i+1}), \eta_\OOO((M')^{i+1})+1 \} = \{\eta_\OOO((M')^{j}), \eta_\OOO((M')^{j})+1 \},
\end{multline*}
and combining it with \cref{eq:mapy22}, we obtain that there exists an integer $l \in \{1,2\}$ such that
 \[\eta_\OOO(M_l) = 2 - \eta_\OOO(M). \]
In other words, there exists $\epsilon \in \{0,1\}$ such that $M_l = \tau_i^\epsilon M'$, and we define:
\[ \si_\OOO(M) := \tau_i^\epsilon M'.\]
Now, it is straightforward from the construction that
$\si_\OOO(\si_\OOO(M))$ is again of the same form. That is
$\si_\OOO(\si_\OOO(M)) := \tau_i^\epsilon \si_\OOO(M)'$, where
$\si_\OOO(M)'$ is a map given by \cref{lem:order} \ref{item:lemma3}. Thus,
\[ \si_\OOO(\si_\OOO(M)) = \tau_i^\epsilon (\tau_i^\epsilon M')' = \tau_i^{2\epsilon} (M')' = M,\]
where the last two equalities are clear from the construction of $M'$
given in the proof of \cref{lem:order} \ref{item:lemma3} -- see \cref{fig:hard}.

Since there are no other cases, we have constructed the required involution, which finishes the proof.
\end{proof}

\medskip

Now, we have all necessary ingredients to prove \cref{theo:UnicellularLowGenera}.

\medskip

\begin{proof}[Proof of \cref{theo:UnicellularLowGenera}]
Let $\OOO$ be an orientation of all rooted maps and let us fix a positive integer $n$ and partitions $\mu, \nu \vdash n$ such that $\ell(\mu) + \ell(\nu) = n-3$. Thanks to \cref{prop:a-ki} we know that $\left(a_{\eta_\OOO}\right)^{(n)}_{\mu,\nu;2}$ is the polynomial in $\beta$ of the following form: 
\[  \left(a_{\eta_\OOO}\right)^{(n)}_{\mu,\nu;2}(\beta) := \sum_{M \in \widetilde{\M}^{(n)}_{\mu,\nu;2}} \beta^{\eta_\OOO(M)} = a + b \beta + c \beta^2,\]
where $a,b,c$ are nonnegative integers.
Moreover, \cref{lem:2handles} gives the following equality 
\[ \sum_{M \in \widetilde{\M}^{(n)}_{\mu,\nu;2}} \beta^{\eta_\OOO(M)} = \sum_{M \in \widetilde{\M}^{(n)}_{\mu,\nu;2}} \beta^{2-\eta_\OOO(M)}.\]
Hence 
\[\left(a_{\eta_\OOO}\right)^{(n)}_{\mu,\nu;2}(\beta) = a + b \beta + a \beta^2.\]
Finally, \cref{lem:LaCroixPolynomials} says that $\left(a_{\eta_\OOO}\right)^{(n)}_{\mu,\nu;2}(-1) = 0$. Thus, there exists a positive integer $\left(\widetilde{a_{\eta_\OOO}}\right)^{(n)}_{\mu,\nu;2}$ such that
\[ \left(a_{\eta_\OOO}\right)^{(n)}_{\mu,\nu;2}(\beta)
=\left(\widetilde{a_{\eta_\OOO}}\right)^{(n)}_{\mu,\nu;2}\cdot (1+\beta)^2. \]
\cref{cor:EtaDegree} asserts that for all positive integers $n$ and
partitions $\mu,\nu \vdash n$ such that $\ell(\mu) + \ell(\nu) > n -
3$ the set $\widetilde{M}^{(n)}_{\mu,\nu;i}$ is empty for $i \geq 2$.
Thus, \cref{prop:a-ki} and \cref{lem:LaCroixPolynomials} give that for
any positive integer $n$, partitions $\mu,\nu \vdash n$ such that
$\ell(\mu) + \ell(\nu) > n - 3$, and a nonnegative integer $i$, there exists a positive integer $\left(\widetilde{a_{\eta_\OOO}}\right)^{(n)}_{\mu,\nu;i}$ such that
\[ \left(a_{\eta_\OOO}\right)^{(n)}_{\mu,\nu;i}(\beta) =\left(\widetilde{a_{\eta_\OOO}}\right)^{(n)}_{\mu,\nu;i}(1+\beta)^i.\]
Plugging it into \cref{eq:HPolynomialForm}, one has the following expression
\[ \left(H_{\eta_\OOO}\right)_{\mu,\nu}^{(n)}(\beta) = \sum_{0 \leq i \leq [g/2]}\left(\widetilde{a_{\eta_\OOO}}\right)^{(n)}_{\mu,\nu;i}\beta^{g-2i}(\beta+1)^i,\]
where $g = n+1 - \left(\ell(\mu)+\ell(\nu)\right) \leq 4$, and the above equation involves at most three coefficients $\left(\widetilde{a_{\eta_\OOO}}\right)^{(n)}_{\mu,\nu;i}$, where $i \in \{0,1,2\}$.
Notice that \cref{lem:expansion} gives a similar expression for quantities $h_{\mu,\nu}^{(n)}(\beta)$:
\[ h_{\mu,\nu}^{(n)}(\beta) = \sum_{0 \leq i \leq [g/2]}a_{\mu,\nu;i}^{(n)} \beta^{g-2i}(\beta+1)^i,\]
and, again, for $g \leq 4$, it involves at most three coefficients $a_{\mu,\nu;i}^{(n)}$, where $i \in \{0,1,2\}$. 
By \cref{theo:UnhandledOnefaceMaps} we know that
\[ h_{\mu,\nu}^{(n)}(\beta) = \left(H_{\eta_\OOO}\right)_{\mu,\nu}^{(n)}(\beta)\]
for $\beta \in \{-1,0,1\}$. So for a fixed positive integer $n$ and
partitions $\mu,\nu \vdash n$ such that $\ell(\mu) + \ell(\nu) \geq n
- 3$ it gives rise to a system of three equations with at most three
indeterminates. It is easy to check that this system is non-degenerate. Thus, it has a unique solution. In other words $a_{\mu,\nu;i}^{(n)} = \left(\widetilde{a_{\eta_\OOO}}\right)^{(n)}_{\mu,\nu;i}$ for all nonnegative integers $i$, which finishes the proof.
\end{proof}

\section{Concluding remarks}
\label{sec:conclude}

We are going to finish this paper by posing some natural questions and
remarks related to the combinatorial side of $b$-conjecture.

\subsection{Removing edges in a different order}
Note that the function $\eta$ given by \cref{def:MeasureOfNonOrient}
is built recursively by the root-deletion procedure, which gives a
natural order on the set of edges of a given map $M$. One can wonder
if there are some other, natural ways to define an order on $E(M)$
which have chances to give an affirmative answer to \cref{conj:BConj}
using the statistic $\eta$  as in \cref{def:MeasureOfNonOrient}, but with
respect to the considered order. For instance, the author of this paper together
with F\'eray and \'Sniady studied \cite{DolegaFeraySniady2014} a problem of
understanding a combinatorial structure of Jack characters, already
mentioned in \cref{subsec:related}. We proved that a
similarly defined ``measure of non-orientability'', but considered with
respect to the uniform random order on $E(M)$, has many desired properties. Chapuy
and the author of this paper constructed in \cite{ChapuyDolega2015}
a certain directed graph associated with a bipartite quandrangulation
$\qq$ (that is a map with all faces of degree $4$), called the Dual
Exploration Graph (DEG, for short), which is visiting all faces of
$\qq$ in some particular order. Since maps (not necessarily bipartite)
with $n$ edges are in a natural bijection with bipartite
quadrangulations with $n$ faces and edges of a given map correspond
to faces of an associated quadrangulation, DEG defines also an order
on the set of edges of a given map and La Croix suggested
\cite{LaCroixPrivate} to use this order to define a measure of
non-orientability $\eta$ with respect to it. We did not study
combinatorial properties of the statistic $\eta$ defined in this way and we leave open the problem wether it gives the correct answer for $b$-conjecture.

\subsection{Unhandled maps and evaluation at \texorpdfstring{$b=-1$}{b=-1}}

\cref{cor:EtaDegree} suggests that unhandled maps are of special
interest: indeed, for any measure of non-orientability $\eta$ and for
any partitions $\mu,\nu,\tau$ of a positive integer $n$, the
top-degree coefficient in the polynomial
$\left(H_\eta\right)_{\mu,\nu}^\tau(\beta)$ is given by unhandled maps
of type $(\mu,\nu;\tau)$. In particular this top-degree part does not
depend on the choice of $\eta$, but it does depend on the order of
edges we are removing from a given map. Moreover, \cref{eq:Basis}
ensures that in the case of one-part partition $\tau = (n)$, the
top-degree coefficient in the polynomial
$\left(H_\eta\right)_{\mu,\nu}^\tau(\beta)$ coincides, up to a sign, with the evaluation of this polynomial in $\beta = -1$. We can also prove that for any partitions $\mu,\nu$ of a positive integer $n \geq 2$ and for any $l \in [n-1]$ the top-degree coefficient in the polynomial  $\left(H_{\eta_\OOO}\right)_{\mu,\nu}^{(n-l,l)}(\beta)$ coincides with the evaluation of this polynomial in $\beta = -1$ for any set $\OOO$ of orientations (in fact, the main idea of the proof in the case when $\ell(\mu) + \ell(\nu) \geq n - 3$ was given in \cref{lem:order}, and the general case is almost the same). Thus, there are natural questions:

\begin{question}
\label{question}
Is it true that for any measure of non-orientability $\eta$, for any positive integer $n$, and for any partitions $\mu, \nu, \tau \vdash n$ the following equality holds true:
\[ \left(H_\eta\right)_{\mu,\nu}^\tau(-1) = \#\left(\widetilde{M_\eta}\right)_{\mu,\nu;0}^\tau?\]
\end{question}

\begin{question}
\label{question2}
Is it true that for any measure of non-orientability $\eta$, for any positive integer $n$, and for any partitions $\mu, \nu, \tau \vdash n$ the following equality holds true:
\[ \left(H_\eta\right)_{\mu,\nu}^\tau(-1) = h_{\mu,\nu}^\tau(-1)?\]
\end{question}

Another interesting direction of the research initiated in this paper
is an understanding of the combinatorial structure of an unhandled map
of type $(\mu,\nu;\tau)$ for arbitrary partitions $\mu,\nu,\tau$ of a
positive integer $n$. Note that the set of unhandled maps of a given
type is not rooted invariant. However, \cref{prop:marginalsum}
together with \cref{ptheo:marginalSumEta} imply that unicellular
unhandled maps with the black vertex distribution $\mu$ and the white
vertex distribution $\nu$ are in a bijection with \emph{orientable
  maps} with the black vertex distribution $\mu$, the white vertex
distribution $\nu$, and the \emph{arbitrary} face distribution, which
clearly \emph{are rooted invariant} (one can even use a proof of
\cref{ptheo:marginalSumEta} to construct such a bijection recursively). Therefore one can ask the following question:

\begin{question}
Is it true that for any given type $(\mu,\nu;\tau)$ of unhandled maps,
there exists some class of maps (orientable?), with the black vertex
distribution $\mu$, the white vertex
distribution $\nu$, and possibly some additional data (labeling faces?), which is rooted invariant, and which is in some natural bijection with the corresponding set of unhandled maps? 
\end{question}

Finally, one can refine \cref{question} by asking:

\begin{question}
Is it true that for any measure of non-orientability $\eta$, for any positive integer $n$, and for any partitions $\mu, \nu, \tau \vdash n$ the following equality holds true:
\begin{align*} \left(H_\eta\right)_{\mu,\nu}^\tau(\beta) = \begin{cases}\sum_{0 \leq i \leq [g/2]} \widetilde{a_{\mu,\nu;i}^\tau} \beta^{g-2i}(\beta+1)^i &\text{ for } \ell(\mu)+\ell(\nu)+\ell(\tau) \leq 2 + n, \\
0 &\text{ otherwise }; \end{cases}
\end{align*}
where $g := 2 + n - \left(\ell(\mu)+\ell(\nu)+\ell(\tau)\right)$ and $2^i \widetilde{a_{\mu,\nu;i}^\tau}  =\#\left(\widetilde{M_\eta}\right)_{\mu,\nu;i}^\tau$?
\end{question} 

We leave all these questions wide open for future research.

\section*{Acknowledgments}

I thank Valentin F\'eray and Piotr \'Sniady for several years of collaboration
on topics related to the current paper. I also thank Michael La Croix for a very interesting discussion concerning 
the $b$-Conjecture. English proofreading service was financed by
\emph{Narodowe Centrum Nauki}, grant number
2014/15/B/ST1/00064.

\bibliographystyle{amsalpha}

\bibliography{biblio2015}

\def\cprime{$'$}
\providecommand{\bysame}{\leavevmode\hbox to3em{\hrulefill}\thinspace}
\providecommand{\MR}{\relax\ifhmode\unskip\space\fi MR }
\providecommand{\MRhref}[2]{%
  \href{http://www.ams.org/mathscinet-getitem?mr=#1}{#2}
}
\providecommand{\href}[2]{#2}
\begin{thebibliography}{{La{ }}15}

\bibitem[BJ07]{BrownJackson2007}
D.~R.~L. Brown and D.~M. Jackson, \emph{{A rooted map invariant,
  non-orientability and {J}ack symmetric functions}}, J. Combin. Theory Ser. B
  \textbf{97} (2007), no.~3, 430--452. \MR{2305897 (2007m:05225)}

\bibitem[CD17]{ChapuyDolega2015}
G.~Chapuy and M.~Do{\l}{\k{e}}ga, \emph{A bijection for rooted maps on general
  surfaces}, J. Combin. Theory Ser. A \textbf{145} (2017), 252--307.
  \MR{3551653}

\bibitem[CJ{\'S}17]{CzyzewskaJankowskaSniady2016}
A.~Czy\.zewska-Jankowska and P.~{\'S}niady, \emph{Bijection between oriented
  maps and weighted non-oriented maps}, Electron. J. Combin. \textbf{24}
  (2017), no.~3, Paper 3.7, 34.

\bibitem[DF16]{DolegaFeray2014}
M.~Do{\l}{\k{e}}ga and V.~F\'eray, \emph{Gaussian fluctuations of {Y}oung
  diagrams and structure constants of {J}ack characters}, Duke Math. J.
  \textbf{165} (2016), no.~7, 1193--1282. \MR{3498866}

\bibitem[DF17]{DolegaFeray2016}
M.~Do{\l}{\k e}ga and V.~F{\'e}ray, \emph{{Cumulants of {J}ack symmetric
  functions and $b$-conjecture}}, Trans. Amer. Math. Soc. (2017),
  doi:10.1090/tran/7191.

\bibitem[DFS14]{DolegaFeraySniady2014}
M.~Do{\l}{\k{e}}ga, V.~F\'eray, and P.~\'Sniady, \emph{Jack polynomials and
  orientability generating series of maps}, S\'em. Lothar. Combin. \textbf{70}
  (2014), Art. B70j, 50. \MR{3378809}

\bibitem[Eyn16]{Eynard:book}
B.~Eynard, \emph{Counting surfaces}, Progress in Mathematical Physics, vol.~70,
  Birkh\"auser/Springer, [Cham], 2016, CRM Aisenstadt chair lectures.
  \MR{3468847}

\bibitem[GHJ01]{GouldenHarerJackson2001}
I.~P. Goulden, J.~L. Harer, and D.~M. Jackson, \emph{{A geometric
  parametrization for the virtual {E}uler characteristics of the moduli spaces
  of real and complex algebraic curves}}, Trans. Amer. Math. Soc. \textbf{353}
  (2001), no.~11, 4405--4427 (electronic). \MR{1851176 (2002g:14035)}

\bibitem[GJ96a]{GouldenJackson1996}
I.~P. Goulden and D.~M. Jackson, \emph{{Connection coefficients, matchings,
  maps and combinatorial conjectures for {J}ack symmetric functions}}, Trans.
  Amer. Math. Soc. \textbf{348} (1996), no.~3, 873--892. \MR{1325917
  (96m:05196)}

\bibitem[GJ96b]{GouldenJackson1996a}
\bysame, \emph{{Maps in locally orientable surfaces, the double coset algebra,
  and zonal polynomials}}, Canad. J. Math. \textbf{48} (1996), no.~3, 569--584.
  \MR{1402328 (97h:05051)}

\bibitem[JV90]{JacksonVisentin1990}
D.~M. Jackson and T.~I. Visentin, \emph{{A character-theoretic approach to
  embeddings of rooted maps in an orientable surface of given genus}}, Trans.
  Amer. Math. Soc. \textbf{322} (1990), no.~1, 343--363. \MR{1012517
  (91b:05093)}

\bibitem[KS97]{KnopSahi1997}
F.~Knop and S.~Sahi, \emph{A recursion and a combinatorial formula for {J}ack
  polynomials}, Invent. Math. \textbf{128} (1997), no.~1, 9--22. \MR{1437493}

\bibitem[KV16]{KanunnikovVassilieva2016}
A.~L. Kanunnikov and E.~A. Vassilieva, \emph{On the matchings-{J}ack conjecture
  for {J}ack connection coefficients indexed by two single part partitions},
  Electron. J. Combin. \textbf{23} (2016), no.~1, Paper 1.53, 30. \MR{3484758}

\bibitem[{La{ }}09]{LaCroix2009}
M.~A. {La{ C}roix}, \emph{{The combinatorics of the {J}ack parameter and the
  genus series for topological maps}}, Ph.D. thesis, University of Waterloo,
  2009.

\bibitem[{La{ }}15]{LaCroixPrivate}
\bysame, Private communication, 2015.

\bibitem[Las08]{Lassalle2008a}
M.~Lassalle, \emph{{A positivity conjecture for {J}ack polynomials}}, Math.
  Res. Lett. \textbf{15} (2008), no.~4, 661--681. \MR{2424904}

\bibitem[Las09]{Lassalle2009}
\bysame, \emph{{Jack polynomials and free cumulants}}, Adv. Math. \textbf{222}
  (2009), no.~6, 2227--2269. \MR{2562783}

\bibitem[LZ04]{LandoZvonkin2004}
S.~K. Lando and A.~K. Zvonkin, \emph{{Graphs on surfaces and their
  applications}}, {Encyclopaedia of Mathematical Sciences}, vol. 141,
  Springer-Verlag, Berlin, 2004, With an appendix by Don B. Zagier. \MR{2036721
  (2005b:14068)}

\bibitem[Mac95]{Macdonald1995}
I.~G. Macdonald, \emph{{Symmetric functions and {H}all polynomials}}, second
  ed., {Oxford Mathematical Monographs}, The Clarendon Press Oxford University
  Press, New York, 1995, With contributions by A. Zelevinsky, Oxford Science
  Publications. \MR{1354144}

\bibitem[Sch15]{Schaeffer2015}
Gilles Schaeffer, \emph{Planar maps}, Handbook of enumerative combinatorics,
  Discrete Math. Appl. (Boca Raton), CRC Press, Boca Raton, FL, 2015,
  pp.~335--395. \MR{3409346}

\bibitem[{\'S}ni15a]{SniadyPrivate}
P.~{\'S}niady, Private communication, 2015.

\bibitem[{\'S}ni15b]{Sniady2015}
\bysame, \emph{{Top degree of {J}ack characters and enumeration of maps}},
  Preprint arXiv:1506.06361, 2015.

\bibitem[Sta89]{Stanley1989}
R.~P. Stanley, \emph{{Some combinatorial properties of {J}ack symmetric
  functions}}, Adv. Math. \textbf{77} (1989), no.~1, 76--115. \MR{1014073
  (90g:05020)}

\bibitem[Vas15]{Vassilieva2014}
E.~A. Vassilieva, \emph{Polynomial properties of {J}ack connection coefficients
  and generalization of a result by {D}\'enes}, J. Algebraic Combin.
  \textbf{42} (2015), no.~1, 51--71. \MR{3365593}

\end{thebibliography}

\end{document}